\documentclass[10pt, a4paper]{article}
\usepackage{amsfonts}
\usepackage{mathrsfs}
\usepackage{latexsym}
\usepackage{xy}
\usepackage{amsfonts,amsmath,amssymb,amsthm}
\usepackage{color}

\xyoption{all}

\newcommand{\bcen}{\begin{center}}     \newcommand{\ecen}{\end{center}}
\newcommand{\bay}{\begin{array}}      \newcommand{\eay}{\end{array}}
\newcommand{\beq}{\begin{eqnarray*}}      \newcommand{\eeq}{\end{eqnarray*}}

\def\ot{\otimes}

\def\ch{\mathrm{char}}

\def\dim{\mathrm{dim}}
\def\dv{\underline{\mathrm{dim}}}
\def\dm{\underline{\underline{\mathrm{dim}}}}

\def\End{\mathrm{End}}

\def\Ext{\mathrm{Ext}}

\def\Hom{\mathrm{Hom}}

\def\id{\mathrm{id}}

\def\inj{\mathrm{inj}}

\def\mod{\mathrm{mod}}
\def\Mod{\mathrm{Mod}}

\def\op{\mathrm{op}}
\def\ot{\otimes}

\def\per{\mathrm{per}}
\def\Per{\mathrm{Per}}
\def\proj{\mathrm{proj}}

\def\rad{\mathrm{rad}}
\def\rank{\mathrm{rank}}

\def\RHom{\mathrm{RHom}}

\def\tr{\mathrm{tr}}
\def\Tr{\mathrm{Tr}}
\def\tv{\underline{\mathrm{tr}}}
\def\tm{\underline{\underline{\mathrm{tr}}}}

\def\Tor{\mathrm{Tor}}

\def\Z{\mathbb{Z}}

\begin{document}

\newtheorem{theorem}{Theorem}
\newtheorem{proposition}{Proposition}
\newtheorem{lemma}{Lemma}
\newtheorem{corollary}{Corollary}
\newtheorem{remark}{Remark}
\newtheorem{example}{Example}
\newtheorem{definition}{Definition}
\newtheorem*{conjecture}{Conjecture}
\newtheorem{question}{Question}

\title{Hirzebruch-Riemann-Roch and Lefschetz type formulas for finite dimensional algebras}

\author{Yang Han}

\date{\footnotesize KLMM, Academy of Mathematics and Systems Science,
Chinese Academy of Sciences, \\ Beijing 100190, China. \\ School of Mathematical Sciences, University of
Chinese Academy of Sciences, \\ Beijing 100049, China.\\ E-mail: hany@iss.ac.cn}

\maketitle

\begin{abstract}
The Hirzebuch-Riemann-Roch (HRR) and Lefschetz type formulas for finite dimensional elementary algebras of finite global dimension are explicitly given.
They have cohomological, homological, Hochschild cohomological and Hochschild homological four versions,
and module, bimodule, module complex and bimodule complex four levels.
For this, the dimension matrix of a bimodule (complex) and the trace matrix of a bimodule (complex) endomorphism are introduced.
It is shown that Shklyarov pairing, Chern character and Hattori-Stallings trace can be concretely expressed by Cartan matrix, dimension vector and trace vector in this situation. Furthermore, the HRR and Lefschetz type formulas for finite dimensional elementary algebras of finite global dimension and dg algebras are compared.
\end{abstract}

\medskip


\medskip



\section{Introduction}

Throughout this paper, $k$ is a fixed field,
and all algebras are assumed to be associative $k$-algebras with identity.

Recently various Hirzebruch-Riemann-Roch (HRR) and Lefschetz type formulas have been formulated
in derived commutative or noncommutative algebraic geometry \cite{Cal05,CalWil10,CisTab14,Lun12,Mar09,Pet13,Pol14,PolVai12,Ram09,Shk13}.
The main ingredients in HRR type formula (or Cardy condition \cite{CalWil10}) and Lefschetz type formula (or generalized HRR type formula \cite{Ram09}, baggy Cardy condition \cite{CalWil10}) are Chern character (or Euler character \cite{Shk13,Lun12}, Chern class, Euler class), Hattori-Stallings trace (or twisted Chern character \cite{Ram09}, boundary-bulk map \cite{PolVai12}, Hochschild class \cite{Pet13}), and Shklyarov pairing (or pairing \cite{Shk13,Pet13}, Mukai pairing \cite{CalWil10}).
Usually it is difficult to calculate explicitly these ingredients and provide concrete HRR and Lefschetz type formulas, just as what mentioned in \cite{Shk13,PolVai12,Pol14}.

In this paper, we study the HRR and Lefschetz type formulas for finite dimensional elementary algebras of finite global dimension.
Recall that a finite dimensional algebra $A$ is {\it elementary} if the factor algebra $A/\rad A$ of $A$ modulo its Jacobson radical $\rad A$ is isomorphic to the direct product of finitely many copies of $k$, or equivalently, $A\cong kQ/I$ where $Q$ is a finite quiver and $I$ is an admissible ideal of the path algebra $kQ$ (Ref. \cite[Chapter III, Theorem 1.9]{AusReiSma95}).
For a finite dimensional elementary algebra of finite global dimension, we can work out all ingredients in the HRR and Lefschetz type formulas.
Indeed, its Hochschild homology is concentrated in degree zero and of $k$-dimension the rank of its Grothendieck group (Proposition~\ref{Proposition-HH-GloDimFinAlg}),
the matrix of its Shklyarov pairing under the canonical bases is the transpose of its Cartan matrix (Proposition~\ref{Proposition-Pairing-CartanMatrix}),
Chern characters can be concretely expressed by dimension vectors and the Cartan matrix of the algebra (Proposition~\ref{Proposition-ChernChar-DimVect}),
and Hattori-Stallings traces can be concretely expressed by trace vectors and the Cartan matrix of the algebra (Proposition~\ref{Proposition-HochClass-TrCartan}).
Far beyond these, we will give explicitly the HRR and Lefschetz type formulas for finite dimensional elementary algebras of finite global dimension,
which have cohomological, homological, Hochschild cohomological and Hochschild homological four versions,
and module, bimodule, module complex and bimodule complex four levels (Theorem 1--4,6--9).
All these HRR (resp. Lefschetz) type formulas are essentially equivalent.
Unlike those in other literatures, our HRR type formulas are identities in matrix additive groups over $\mathbb{Z}$ but not the base field $k$.
The cohomological HRR type formula on module level is just Ringel's formula in
\cite[Lemma 2.4]{Rin84}.
The Hochschild cohomological HRR type formula on module level generalizes Happel's formula in \cite[Theorem 2.2]{Hap97} from Hochschild cohomology to Hochschild cohomology with coefficients.
The Hochschild homological HRR type formula on module level generalizes Zhang-Liu's formula in \cite[Theorem]{ZhaLiu97} from Hochschild homology to Hochschild homology with coefficients.
The HRR type formulas for dg algebras were given by Shklyarov in
\cite[Theorem 1.2, Theorem 1.3 and Proposition 4.4]{Shk13}.
They are identities in the base field $k$.
In the case of $\ch k=0$, for a very important class of finite dimensional elementary algebras of finite global dimension --- triangular algebras ($=$ directed algebras in \cite{Shk13} $\neq$ directed algebras in \cite{Rin84}), Shklyarov deduced Ringel's formula from his formula \cite[5.1]{Shk13}. Recall that a finite dimensional elementary algebra is {\it triangular} if it is isomorphic to a bound quiver algebra $kQ/I$ where the quiver $Q$ is acyclic.
We will show that, in the case of $\ch k=0$, for all finite dimensional elementary algebras of finite global dimension,
Shklyarov's formula in \cite[Theorem 1.2 and Theorem 1.3]{Shk13} is just the cohomological HRR type formula on complex level in Theorem~\ref{Theorem-HRR-Complex} (1)
which is equivalent to Ringel's formula,
and Shklyarov's formula in \cite[Proposition 4.4]{Shk13} is just the homological HRR type formula on complex level in Theorem~\ref{Theorem-HRR-Complex} (2) (See \S \ref{Subsection-HRR-dga}).
The Lefschetz type formulas for dg algebras were given by Petit in \cite[Proposition 5.5 and Theorem 5.6]{Pet13}.
We will prove that, the homological Lefschetz type formula on complex level in Theorem~\ref{Theorem-Lefschetz-Complex} (2)
is just Petit's formula in \cite[Theorem 5.6]{Pet13} restricted to finite dimensional elementary algebras of finite global dimension,
and the Hochschild homological Lefschetz type formula on complex level in Theorem~\ref{Theorem-Lefschetz-Complex} (4)
is just Petit's formula in \cite[Proposition 5.5]{Pet13} restricted to finite dimensional elementary algebras of finite global dimension
(See \S \ref{Subsection-Lefschetz-dga}).

The paper is organized as follows:
In section 2, we will introduce the dimension matrix of a bimodule (complex). Then we will provide various versions of HRR type formulas on various levels for finite dimensional elementary algebras of finite global dimension.
Moreover, we will give the matrix of Shklyarov pairing under canonical bases,
and express Chern characters by dimension vectors and the Cartan matrix of the algebra.
Furthermore, we compare the HRR type formulas for finite dimensional elementary algebras of finite global dimension with the HRR type formulas for dg algebras, i.e., Shklyarov's formulas. In section 3, we will introduce the trace matrix of a bimodule (complex) endomorphism. Then we will give various versions of Lefschetz type formulas on various levels for finite dimensional elementary algebras of finite global dimension.
Moreover, we will express Hattori-Stallings traces by trace vectors and the Cartan matrix of the algebra.
Furthermore, we compare the Lefschetz type formulas for finite dimensional elementary algebras of finite global dimension with the Lefschetz type formulas for dg algebras, i.e., Petit's formulas.

We refer to \cite{AusReiSma95} for representation theory of finite dimensional algebras,
to \cite{Kel94,Kel06} for dg algebras and dg categories, and to \cite{Wei94} for homological algebra.
By convention, a complex $X$ is both a cochain complex $(X^l,d^l)_{l\in\Z}$ and a chain complex $(X_l,d_l)_{l\in\Z}$ but the homogeneous component $X^l=X_{-l}$ and the differential $d^l=d_{-l}$ for all $l\in\Z$.
We denote by $R^{n\times m}$ the set of all $n\times m$ matrices with entries in a ring $R$.
Moreover, $\dim:=\dim_k$, $\ot:=\ot_k$ and $(-)^*:=\Hom_k(-,k)$.

\section{HRR type formulas for finite dimensional algebras}

In this section, we will give the HRR type formulas for finite dimensional elementary algebras of finite global dimension and compare them with the HRR type formulas for dg algebras.

\subsection{Dimension matrices}

In order to formulate the HRR type formulas for finite dimensional elementary algebras of finite global dimension, we introduce the dimension matrix of a bimodule (complex) which generalizes both the Cartan matrix of an algebra and the dimension vector of a module.

\bigskip

\noindent{\bf Cartan matrices of algebras.} Let $A$ be a finite dimensional elementary algebra and
$\{e_1,\cdots,e_n\}$ a complete set of orthogonal primitive idempotents of $A$.
Then $\{e_1A,\cdots,e_nA\}$ is a complete set of representatives of isomorphism classes of indecomposable projective right $A$-modules.
The {\it Cartan matrix} $C_A$ of the algebra $A$ is the integer-valued $n\times n$ matrix $(c_{ij})$ with entries
$c_{ij}:=\dim\Hom_A(e_iA,e_jA)=\dim e_jAe_i, 1\le i,j\le n$.
If $A$ is a finite dimensional elementary algebra of finite global dimension
then $\det C_A=\pm 1$, and thus $C_A$ is invertible and its inverse matrix $C_A^{-1}$ is also an integer-valued $n\times n$ matrix.
In this situation, the {\it Ringel form} of the algebra $A$ is
$\langle -,- \rangle_A : \mathbb{Z}^n\times\mathbb{Z}^n\to\mathbb{Z}, (x,y)\mapsto x^T C_A^{-T}y$,
where the elements in $\Z^n$ are column vectors and $T$ denotes transpose.
The {\it Coxeter matrix} of the algebra $A$ is $\Phi_A:=-C_A^{-T}\cdot C_A$.
All these concepts play quite important roles in representation theory of finite dimensional algebras \cite{AusReiSma95,Rin84}.

The following lemma gives the Cartan matrices of tensor product algebra, opposite algebra and enveloping algebra, and generalizes \cite[Lemma 2.1 (ii)]{Hap97}.

\begin{lemma} \label{Lemma-CartanMatrix-AlgebraTensorOp}
Let $A$ and $B$ be two finite dimensional elementary algebras,
and $\{e_1,\cdots,e_n\}$ and $\{f_1,\cdots,f_m\}$ be complete sets of orthogonal primitive idempotents of $A$ and $B$ respectively. Then the following three identities hold:

{\rm (1)} $C_{B\ot A}=C_B\ot C_A$.

{\rm (2)} $C_{A^\op}=C_A^T$, where $A^\op$ is the opposite algebra of $A$.

{\rm (3)} $C_{A^e}=C_A^T \ot C_A$, where $A^e:=A^\op\ot A$ is the enveloping algebra of $A$.
\end{lemma}

\begin{proof}
(1) With respect to the canonical complete set $\{f_1\ot e_1,\cdots,f_1\ot e_n,\cdots, \linebreak f_m\ot e_1,\cdots,f_m\ot e_n\}$
of orthogonal primitive idempotents of $B\ot A$, we have  \linebreak
$(C_{B\ot A})_{n(j-1)+i,n(j'-1)+i'} = \dim (f_{j'}\ot e_{i'})(B\ot A)(f_j\ot e_i)
= \dim (f_{j'}Bf_j \ot e_{i'}Ae_i) = \dim f_{j'}Bf_j \cdot \dim e_{i'}Ae_i
= (C_B)_{jj'}\cdot (C_A)_{ii'} = (C_B\ot C_A)_{n(j-1)+i,n(j'-1)+i'}$ for all $1\le j,j'\le m$ and $1\le i,i'\le n$.
So $C_{B\ot A}=C_B\ot C_A$.

(2) With respect to the canonical complete set $\{e_1,\cdots,e_n\}$ of orthogonal primitive idempotents of $A^\op$,
we have $(C_{A^\op})_{ij}=\dim e_jA^\op e_i=\dim e_i A e_j=(C_A)_{ji}$ for all $1\le i,j\le n$.
So $C_{A^\op}=C_A^T$.

(3) follows immediately from (1) and (2).
\end{proof}

\noindent{\bf Dimension vectors of modules revised.}
Let $A$ be a finite dimensional elementary algebra and $\{e_1,\cdots,e_n\}$ a complete set of orthogonal primitive idempotents of $A$.
The {\it dimension vector} of a finite dimensional right $A$-module $M$ is the column vector $\dv M:=[\dim Me_1,\cdots,\dim Me_n]\in\Z^n$.
The {\it dimension vector} of a finite dimensional left $A$-module $N$ is the row vector $\dv N:=(\dim e_1N,\cdots,\dim e_nN)\in\Z^n$.
Any finite dimensional right $A$-module $M$ can be viewed as a left $A^\op$-module naturally,
but $\dv M_A =(\dv\ _{A^\op}M)^T$.

\begin{remark}{\rm
In representation theory of finite dimensional algebras, the dimension vectors of both left modules and right modules are row vectors.
Here we change this convention so that they are compatible with the definition of dimension matrix of a bimodule below.
}\end{remark}

\noindent{\bf Dimension matrices of bimodules.}
Let $A$ and $B$ be finite dimensional elementary algebras,
and $\{e_1,\cdots,e_n\}$ and $\{f_1,\cdots,f_m\}$ complete sets of orthogonal primitive idempotents of $A$ and $B$ respectively.
The {\it dimension matrix} or {\it Cartan matrix} of a finite dimensional $B$-$A$-bimodule $M$
is the integer-valued $n\times m$ matrix $\dm M = C_M:=(c_{ij})$
where $c_{ij}:= \dim\Hom_A(e_iA,f_jM)=\dim f_jMe_i$ for all $1\le i\le n$ and $1\le j\le m$.

\begin{remark} \label{Remark-CartanMatrix-LeftRightMod} {\rm
(1) The Cartan matrix $C_A$ of a finite dimensional elementary algebra $A$ is just the dimension matrix $\dm {_AA_A}$ or Cartan matrix $C_{_AA_A}$ of the $A$-bimodule $A$.

(2) A finite dimensional right $A$-module $M$ can be viewed as a $k$-$A$-bimodule naturally,
and $\dv M_A = \dm\ {_kM_A}$.
A finite dimensional left $A$-module $N$ can be viewed as an $A$-$k$-bimodule naturally,
and $\dv {_AN} = \dm\ {_AN_k}$.
In particular, for a finite dimensional $k$-vector space $M$, $\dim M=\dm\ {_kM_k}$.
So the dimension matrix of a finite dimensional bimodule generalizes the dimension of a finite dimensional vector space and the dimension vector of a finite dimensional module.
}\end{remark}

The following lemma distinguishes $\dm{_BM_A}$, $\dv {M_{B^\op\ot A}}$ and $\dv{_{B\ot A^\op}M}$, and generalizes \cite[Lemma 2.1 (i)]{Hap97}.

\begin{lemma} \label{Lemma-DimVector-Bimod}
Let $A$ and $B$ be finite dimensional elementary algebras,
$\{e_1,\cdots,e_n\}$ and $\{f_1,\cdots,f_m\}$ complete sets of orthogonal primitive idempotents of $A$ and $B$ respectively,
and $M$ a finite dimensional $B$-$A$-bimodule. Then

{\rm (1)} $\dv{M_{B^\op\ot A}} = [\dv f_1M,\cdots,\dv f_mM]$
where $\dv f_jM$ is just the $j$-th column of the dimension matrix $\dm M$ of the $B$-$A$-bimodule $M$ for all $1\le j\le m$, i.e.,
$\dv{M_{B^\op\ot A}}$ is the column vectorization of $\dm M$.

{\rm (2)}  $\dv{_{B\ot A^\op}M} = (\dv M_{B^\op\ot A})^T$ is the row vectorization of $(\dm M)^T$.
\end{lemma}

\begin{proof}
(1) With respect to the canonical complete set $\{f_1\ot e_1,\cdots,f_1\ot e_n,\cdots,\linebreak f_m\ot e_1,\cdots,f_m\ot e_n\}$
of orthogonal primitive idempotent of $B^\op\ot A$, we have
$$\begin{array}{ll}
  & \dv{M_{B^\op\ot A}} \\ [2mm]
= & [\dim\ M(f_1\ot e_1),\cdots,\dim\ M(f_1\ot e_n),\cdots,\dim\ M(f_m\ot e_1),\cdots,\dim\ M(f_m\ot e_n)] \\ [2mm]
= & [\dim f_1Me_1,\cdots,\dim f_1Me_n,\cdots,\dim f_mMe_1,\cdots,\dim f_mMe_n] \\ [2mm]
= & [\dv f_1M,\cdots,\dv f_mM].
\end{array}$$

(2) follows from (1).
\end{proof}

We know that the dimension of tensor product of two finite dimensional vector spaces is equal to the product of their dimensions.
More general, we have the following lemma:

\begin{lemma} \label{Lemma-CartanMatrix-ModuleTensor}
Let $A$ and $B$ be finite dimensional elementary algebras,
$\{e_1,\cdots,e_n\}$ and $\{f_1,\cdots,f_m\}$ complete sets of orthogonal primitive idempotents of $A$ and $B$ respectively,
$M$ a finite dimensional right $A$-module, and $N$ a finite dimensional left $B$-module.
Then $\dm(N\ot M) = \dv M\cdot \dv N$.
\end{lemma}

\begin{proof}
We have $(\dm(N\ot M))_{ij} = \dim(f_jN\ot Me_i) = \dim Me_i \cdot \dim f_jN = (\dv M)_i\cdot (\dv N)_j =  (\dv M\cdot \dv N)_{ij}$
for all $1\le i\le n$ and $1\le j\le m$. So $\dm(N\ot M)= \dv M\cdot \dv N$.
\end{proof}

We know that a finite dimensional vector space and its dual space have the same dimension. More general, we have the following lemma:

\begin{lemma} \label{Lemma-CartanMatrix-ModuleDual}
Let $A$ and $B$ be finite dimensional elementary algebras,
$\{e_1,\cdots,e_n\}$ and $\{f_1,\cdots,f_m\}$ complete sets of orthogonal primitive idempotents of $A$ and $B$ respectively,
$M$ a finite dimensional $B$-$A$-bimodule, and $M^*:=\Hom_k(M,k)$ the dual $A$-$B$-bimodule of $M$.
Then $\dm M^* = (\dm M)^T$. In particular, if $M$ is a finite dimensional left or right $A$-module then
$\dv M^* = (\dv M)^T$.
\end{lemma}

\begin{proof}
We have $(\dm M^*)_{ji} = \dim e_iM^*f_j = \dim (f_jMe_i)^* = \dim f_jMe_i = (\dm M)_{ij}$
for all $1\le i\le n$ and $1\le j\le m$. So $\dm M^* = (\dm M)^T$.
\end{proof}

\bigskip

\noindent{\bf Dimension matrices of complexes.}
The {\it (super) dimension} of a bounded complex $M$ of finite dimensional $k$-vector spaces is
the integer $\dim M := \linebreak \sum\limits_{l\in\Z}(-1)^l\ \dim M^l$.
Let $A$ be a finite dimensional elementary algebra,
and $\{e_1,\cdots,\linebreak e_n\}$ a complete set of orthogonal primitive idempotents of $A$.
The {\it (super) dimension vector} of a bounded complex $M$ of finite dimensional right (resp. left) $A$-modules is the column (resp. row) vector $\dv M := \sum\limits_{l\in\Z}(-1)^l\ \dv M^l\in\Z^n$.
Let $B$ be also a finite dimensional elementary algebra,
and $\{f_1,\cdots, f_m\}$ a complete set of orthogonal primitive idempotents of $B$.
The {\it (super) dimension matrix} or {\it (super) Cartan matrix} of a bounded complex $M$ of finite dimensional $B$-$A$-bimodules is the integer-valued $n\times m$ matrix $\dm M = C_M:= \sum\limits_{l\in\Z}(-1)^l\ \dm M^l$.

The following lemma implies that dimension matrix is an additive invariant on the category $B\mbox{\rm -mod-}A$ of finite dimensional $B$-$A$-bimodules.

\begin{lemma} \label{Lemma-DimMatrix-qis}
Let $A$ and $B$ be two finite dimensional elementary algebras,
and $\{e_1,\cdots,e_n\}$ and $\{f_1,\cdots,f_m\}$ complete sets of orthogonal primitive idempotents of $A$ and $B$ respectively.
Then the following two statements hold:

{\rm (1)} For any short exact sequence $0\to M'\to M\to M''\to 0$ in the category $B\mbox{\rm -mod-}A$ of finite dimensional $B$-$A$-bimodules,
$\dm M=\dm M'+\dm M''$.
So $\dm : K_0(B\mbox{\rm -mod-}A) \to \Z^{n\times m}, [M] \mapsto \dm M$, is a group homomorphism.
Here, $K_0(B\mbox{\rm -mod-}A)$ is the Grothendieck group of the abelian category $B\mbox{\rm -mod-}A$.

{\rm (2)} For any bounded complex $M$ of finite dimensional $B$-$A$-bimodules,
$$\dm M := \sum\limits_{l\in\Z}(-1)^l\ \dm M^l = \sum\limits_{l\in\Z}(-1)^l\ \dm H^l(M).$$
So dimension matrix is invariant under quasi-isomorphisms of bounded complexes of finite dimensional bimodules.
\end{lemma}

\begin{proof}
(1) Acting the exact functor $f_jB\ot_B-\ot_AAe_i: B\mbox{\rm -mod-}A \to \mod k$ on the given short exact sequence $0\to M'\to M\to M''\to 0$,
we obtain a short exact sequence $0\to f_jM'e_i\to f_jMe_i\to f_jM''e_i\to 0$ in the category $\mod k$ of finite dimensional $k$-vector spaces.
Thus $(\dm M)_{ij}=(\dm M')_{ij}+(\dm M'')_{ij}$ for all $1\le i\le n$ and $1\le j\le m$.
So $\dm M=\dm M'+\dm M''$.

(2) Let $Z^l(M)$ and $B^l(M)$ be the $l$-cocycle and $l$-coboundary of $M$ respectively.
Then we have short exact sequences $0 \to B^l(M) \to Z^l(M) \to H^l(M) \to 0$
and $0 \to Z^l(M) \to M^l \to B^{l+1}(M) \to 0$ in $B\mbox{\rm -mod-}A$.
By (1), we have $\dm Z^l(M) = \dm B^l(M)+\dm H^l(M)$ and $\dm M^l= \dm Z^l(M) + \dm B^{l+1}(M)$.
So $\dm M^l= \dm B^l(M) + \dm B^{l+1}(M)+\dm H^l(M)$ for all $l\in\Z$.
Thus $\dm M := \sum\limits_{l\in\Z}(-1)^l\ \dm M^l = \sum\limits_{l\in\Z}(-1)^l\ \dm H^l(M).$
\end{proof}

\begin{remark} \label{Remark-DimMatrix-CohomFinDimComplex} {\rm
More general, for a {\it cohomologically finite dimensional} complex $M$ of $B$-$A$-bimodules,
i.e., $\sum\limits_{l\in\Z}\dim H^l(M)<\infty$,
we define its {\it (super) dimension matrix} or {\it (super) Cartan matrix}
$\dm M = C_M := \sum\limits_{l\in\Z}(-1)^l\ \dm H^l(M)$ which is clearly invariant under quasi-isomorphisms of complexes.
Due to Lemma~\ref{Lemma-DimMatrix-qis} (2), this definition extends that for a bounded complex $M$ of finite dimensional $B$-$A$-bimodules.
}\end{remark}

\subsection{HRR type formulas for finite dimensional algebras}

In this subsection, using dimension vector and dimension matrix,
we will give the HRR type formulas on module, bimodule, module complex and bimodule complex four levels
for finite dimensional elementary algebras of finite global dimension.

\bigskip

\noindent{\bf HRR type formulas on module level.}
The following theorem gives cohomological, homological, Hochschild cohomological and Hochschild homological
four versions of HRR type formulas on module level for finite dimensional elementary algebras of finite global dimension.
The cohomological HRR type formula on module level is just Ringel's formula in \cite[Lemma 2.4]{Rin84}.

\begin{theorem} \label{Theorem-HRR-fda}
Let $A$ be a finite dimensional elementary algebra of finite global dimension,
and $\{e_1,\cdots,e_n\}$ a complete set of orthogonal primitive idempotents in $A$.
Then the following four equivalent statements hold:

{\rm (1)} {\rm (Ringel \cite[Lemma 2.4]{Rin84})} For all finite dimensional right $A$-modules $M$ and $N$,
$$\dim(\RHom_A(M,N)) := \sum\limits_{l\ge 0} (-1)^l\ \dim\Ext_A^l(M,N) = \langle \dv M,\dv N \rangle_A$$
where $\langle -,- \rangle_A : \mathbb{Z}^n\times\mathbb{Z}^n\to\mathbb{Z}, (x,y)\mapsto x^T\cdot C_A^{-T}\cdot y$, is the Ringel form of $A$.

{\rm (2)} For all finite dimensional right $A$-module $M$ and finite dimensional left $A$-module $N$,
$$\dim(M\ot^L_AN) := \sum\limits_{l\ge 0} (-1)^l\ \dim\Tor^A_l(M,N) = \langle (\dv N)^T,\dv M \rangle_{A^\op}$$
where $\langle -,- \rangle_{A^\op} : \mathbb{Z}^n\times\mathbb{Z}^n\to\mathbb{Z}, (x,y)\mapsto x^T\cdot C_A^{-1}\cdot y$, is the Ringel form of $A^\op$.

{\rm (3)} For any finite dimensional $A$-bimodule $M$,
$$\dim(\RHom_{A^e}(A,M)) := \sum\limits_{l\ge 0} (-1)^l\ \dim HH^l(A,M) = \tr(C_A^{-T}\cdot \dm M)$$
where $HH^l(A,M):=\Ext_{A^e}^l(A,M)$ is the $l$-th Hochschild cohomology of $A$ with coefficients in $M$.

{\rm (4)} For any finite dimensional $A$-bimodule $M$,
$$\dim(A\ot^L_{A^e}M) := \sum\limits_{l\ge 0} (-1)^l\ \dim HH_l(A,M)  = \tr(C_A^{-1}\cdot \dm M)$$
where $HH_l(A,M):=\Tor^{A^e}_l(A,M)$ is the $l$-th Hochschild homology of $A$ with coefficients in $M$.
\end{theorem}

\begin{remark} \label{Remark-HRR-Mod} {\rm
(1) The leftmost terms of the HRR type formulas on module level in Theorem~\ref{Theorem-HRR-fda} are just the (super) dimensions
of the complexes $\RHom_A(M,N)$, $M\ot^L_AN$, $\RHom_{A^e}(A,M)$ and $A\ot^L_{A^e}M$ in the derived category $\mathcal{D}k$ of $k$
(See Remark~\ref{Remark-DimMatrix-CohomFinDimComplex}).

(2) We do have the identities $\dim(\RHom_A(M,N)) = \langle \dv N,\dv M \rangle_{A^\op}$ and $\dim(M\ot^L_AN) = \langle \dv M,(\dv N)^T \rangle_A.$ Nonetheless, they are not so natural due to Theorem~\ref{Theorem-HRR-Bimod}.

(3) Unlike those in other literatures, the four HRR type formulas in Theorem~\ref{Theorem-HRR-fda} are identities in $\mathbb{Z}$ but not $k$.
In the case of $A=k$, they are just $\dim\Hom_k(M,N)=\dim M\cdot\dim N$, $\dim(M\otimes N)=\dim M\cdot\dim N$,
$\dim\Hom_k(k,M) \linebreak =\dim M$ and $\dim(k\otimes M)=\dim M$, for finite dimensional $k$-vector spaces $M$ and $N$.
}\end{remark}

\begin{proof} (1) is just \cite[Lemma 2.4]{Rin84}, so it is enough to show that the four statements are equivalent.

$(1)\Rightarrow(3):$
By (1), Lemma~\ref{Lemma-DimVector-Bimod} (1) and Lemma~\ref{Lemma-CartanMatrix-AlgebraTensorOp} (3), we have
$$\begin{array}{ll}
& \sum\limits_{l\ge 0} (-1)^l\ \dim HH^l(A,M) \\ [4mm]
= & \sum\limits_{l\ge 0} (-1)^l\ \dim \Ext_{A^e}^l(A,M) \\ [2mm]
\stackrel{(1)}{=} & (\dv A_{A^e})^T \cdot C_{A^e}^{-T} \cdot \dv M_{A^e} \\ [4mm]
\stackrel{\rm 2L}{=} & \begin{pmatrix} \dv e_1A \\ \vdots \\ \dv e_nA \end{pmatrix}^T \cdot (C_A^{-1} \ot C_A^{-T}) \cdot
\begin{pmatrix} \dv e_1M \\ \vdots \\ \dv e_nM \end{pmatrix} \\ [10mm]
= & \begin{pmatrix}\dv e_1A \\ \vdots \\ \dv e_nA\end{pmatrix}^T \cdot
\begin{pmatrix} (C_A^{-1})_{11}\cdot C_A^{-T} & \cdots & (C_A^{-1})_{1n}\cdot C_A^{-T}\\
\vdots & \ddots & \vdots\\ (C_A^{-1})_{n1}\cdot C_A^{-T} & \cdots & (C_A^{-1})_{nn}\cdot C_A^{-T} \end{pmatrix} \cdot
\begin{pmatrix} \dv e_1M \\ \vdots \\ \dv e_nM \end{pmatrix} \\ [10mm]
= & \sum\limits_{1\le i,j\le n}(\dv e_iA)^T\cdot(C_A^{-1})_{ij}\cdot C_A^{-T}\cdot \dv e_jM \\ [4mm]
\stackrel{(1)}{=} & \sum\limits_{1\le i,j\le n} (C_A^{-1})_{ij}\cdot \dim\Hom_A(e_iA,e_jM) \\ [5mm]
= & \sum\limits_{1\le i,j\le n} (C_A^{-1})_{ij}\cdot (\dm M)_{ij} \\ [5mm]
= & \tr(C_A^{-T}\cdot \dm M).
\end{array}$$

$(3)\Rightarrow(4):$
Note that $HH_l(A,M) = H_l(A\ot_{A^e}^LM) \cong H^l((A\ot_{A^e}^LM)^*) \cong H^l(\RHom_{A^e}(A,M^*)) = HH^l(A,M^*)$ for all $l\in\Z$.
By (3) and Lemma~\ref{Lemma-CartanMatrix-ModuleDual}, we have
$$\begin{array}{ll}
& \sum\limits_{l\ge 0} (-1)^l\ \dim HH_l(A,M) \\
= & \sum\limits_{l\ge 0} (-1)^l\ \dim HH^l(A,M^*) \stackrel{(3)}{=} \tr(C_A^{-T}\cdot \dm M^*) \\
\stackrel{\rm L}{=} & \tr(C_A^{-T}\cdot (\dm M)^T) = \tr(\dm M\cdot C_A^{-1}) = \tr(C_A^{-1}\cdot \dm M).
\end{array}$$

$(4)\Rightarrow(2):$
Note that $\Tor^A_l(M,N)=H_l(M\ot^L_AN) \cong H_l(A\ot^L_{A^e}(N\ot M)) = HH_l(A,N\ot M)$ for all $l\in\Z$.
By (4), Lemma~\ref{Lemma-CartanMatrix-ModuleTensor} and Lemma~\ref{Lemma-CartanMatrix-AlgebraTensorOp} (2), we have
$$\begin{array}{ll}
& \sum\limits_{l\ge 0} (-1)^l\ \dim\Tor^A_l(M,N) \\
= & \sum\limits_{l\ge 0} (-1)^l\ \dim HH_l(A,N\ot M) \stackrel{(4)}{=} \tr(C_A^{-1}\cdot \dm(N\ot M)) \\
\stackrel{\rm L}{=} & \tr(C_A^{-1}\cdot \dv M\cdot \dv N) = \tr(\dv N\cdot C_A^{-1}\cdot \dv M) \\ [2mm]
= & \dv N\cdot C_A^{-1}\cdot \dv M \stackrel{\rm L}{=} \langle (\dv N)^T,\dv M \rangle_{A^\op}.
\end{array}$$

$(2)\Rightarrow(1):$ Note that
$$\begin{array}{ll}
\Ext_A^l(M,N) & = H^l(\RHom_A(M,N)) \cong H^l((M\ot^L_AN^*)^*) \\ [2mm]
& \cong H_l(M\ot^L_AN^*) =\Tor^A_l(M,N^*)
\end{array}$$
for all $l\in\Z$.
By (2) and Lemma~\ref{Lemma-CartanMatrix-ModuleDual}, we have
$$\begin{array}{ll}
& \sum\limits_{l\ge 0} (-1)^l\ \dim\Ext_A^l(M,N) \\
= & \sum\limits_{l\ge 0} (-1)^l\ \dim\Tor^A_l(M,N^*) \stackrel{(2)}{=} \dv N^*\cdot C_A^{-1}\cdot \dv M \\
\stackrel{\rm L}{=} & (\dv N)^T\cdot C_A^{-1}\cdot \dv M =(\dv M)^T\cdot C_A^{-T}\cdot \dv N = \langle \dv M,\dv N \rangle_A. \end{array}$$

Therefore, the four statements (1), (2), (3) and (4) are equivalent.
\end{proof}

\bigskip

Taking the $A$-bimodule $M$ in Theorem~\ref{Theorem-HRR-fda} (3) and (4) to be the $A$-bimodule $A$, we obtain the following two corollaries:

\begin{corollary} \label{Corollary-Happel} {\rm (Happel \cite[Theorem 2.2]{Hap97})}
Let $A$ be a finite dimensional elementary algebra of finite global dimension,
and $\Phi_A := -C_A^{-T}\cdot C_A$ the Coxeter matrix of $A$. Then
$$\dim(\RHom_{A^e}(A,A)) := \sum\limits_{l\ge 0} (-1)^l\ \dim HH^l(A) = -\tr \Phi_A.$$
\end{corollary}

\begin{corollary} \label{Corollary-Zhang-Liu} {\rm (Zhang-Liu \cite[Theorem]{ZhaLiu97}) }
Let $A$ be a finite dimensional elementary algebra of finite global dimension, and $n:=\dim A/\rad A=\rank K_0(A)$. Then
$$\dim(A\ot^L_{A^e}A) := \sum\limits_{l\ge 0} (-1)^l\ \dim HH_l(A)=n.$$
\end{corollary}

The following result is subtler than Corollary~\ref{Corollary-Zhang-Liu}.
It is \cite[Proposition 6]{Han06} obtained directly from \cite[Proposition 2.5]{Kel98}.
Alternatively, on one hand, by \cite[5.3]{Shk13}, we have
$HH_0(A) \cong A/[A,A]$ and $HH_l(A)=0$ for all $l\ge 1$.
On the other hand, by \cite[\S 5 Korollar]{Len69}, we can obtain $\rad A = [A,A]$.
Thus $HH_0(A) \cong k^n$ and $HH_l(A)=0$ for all $l\ge 1$.

\begin{proposition} \label{Proposition-HH-GloDimFinAlg} {\rm (Keller \cite[Proposition 2.5]{Kel98})}
Let $A$ be a finite dimensional elementary algebra of finite global dimension, and $n:=\dim A/\rad A=\rank K_0(A)$. Then
$HH_0(A) \cong k^n$ and $HH_l(A)=0$ for all $l\ge 1$.
\end{proposition}

Proposition~\ref{Proposition-HH-GloDimFinAlg} will play crucial roles in determining
the main ingredients in Shklyarov's formulas and Petit's formulas for finite dimensional elementary algebras of finite global dimension.

\bigskip

\noindent{\bf HRR type formulas on bimodule level.}
The following theorem gives the cohomological and homological HRR type formulas on bimodule level for finite dimensional elementary algebras of finite global dimension.
It generalizes at first sight but is essentially equivalent to Theorem~\ref{Theorem-HRR-fda}.

\begin{theorem} \label{Theorem-HRR-Bimod}
Let $A,B$ and $C$ be three finite dimensional elementary algebras, $A$ of finite global dimension,
and $\{e_1,\cdots, e_n\}$, $\{f_1,\cdots,f_m\}$ and $\{g_1,\cdots,g_p\}$ complete sets of orthogonal primitive idempotents of $A$, $B$ and $C$ respectively.
Then the following two equivalent statements hold:

{\rm (1)} For all finite dimensional $B$-$A$-bimodule $M$ and finite dimensional $C$-$A$-bimodule $N$,
$$\dm(\RHom_A(M,N)) := \sum\limits_{l\ge 0} (-1)^l\ \dm \Ext_A^l(M,N) = \langle \dm M,\dm N\rangle_A$$
where $\langle -,- \rangle_A : \Z^{n\times m} \times \Z^{n\times p} \to \Z^{m\times p},
(x,y)\mapsto x^T\cdot C_A^{-T}\cdot y$.

{\rm (2)} For all finite dimensional $B$-$A$-bimodule $M$ and finite dimensional $A$-$C$-bimodule $N$,
$$\dm(M\ot^L_AN) := \sum\limits_{l\ge 0} (-1)^l\ \dm \Tor^A_l(M,N)  = \langle (\dm N)^T,\dm M \rangle_{A^\op}$$
where $\langle -,- \rangle_{A^\op} : \Z^{n\times p} \times \Z^{n\times m} \to \Z^{p\times m},
(x,y)\mapsto x^T\cdot C_A^{-1}\cdot y$.
\end{theorem}

\medskip

\begin{remark} {\rm
Unlike the cohomological and homological HRR type formulas on module level (See Remark~\ref{Remark-HRR-Mod} (2)),
in general, we have no $\dm(\RHom_A(M,N)) = \langle \dm N,\dm M \rangle_{A^\op}$,
since its left hand side is an $m\times p$ matrix but its right hand side is a $p\times m$ matrix.
Similarly, in general, we have no $\dm(M\ot^L_AN) = \langle \dm M,(\dm N)^T \rangle_A$ either.
}\end{remark}

\medskip

\begin{proof}
By Theorem~\ref{Theorem-HRR-fda}, it suffices to show that Theorem~\ref{Theorem-HRR-Bimod} (1) (resp. (2)) holds if and only if so does Theorem~\ref{Theorem-HRR-fda} (1) (resp. (2)).

\medskip

Theorem~\ref{Theorem-HRR-Bimod} (1) $\Leftrightarrow$ Theorem~\ref{Theorem-HRR-fda} (1):

{\it Sufficiency.} It is enough to prove
$$\sum\limits_{l\ge 0}\ (-1)^l\ (\dm \Ext_A^l(M,N))_{ij}
= ((\dm M)^T\cdot C_A^{-T}\cdot \dm N)_{ij}$$
for all $1\le i\le m$ and $1\le j\le p$.
For this, let $P_M$ be any projective resolution of the $B$-$A$-bimodule $M$.
Since $f_iB$ is a projective right $B$-module, $f_iB\ot_B P_M$ is a projective resolution of the right $A$-module $f_iB\ot_B M\cong f_iM$.
Note that $g_jC\ot_C-\ot_BBf_i$ is an exact functor
from the category $C\mbox{\rm -Mod-}B$ of $C$-$B$-bimodules to the category $\Mod k$ of $k$-vector spaces.
Thus $g_j\Ext_A^l(M,N)f_i
\cong g_jC\ot_CH^l(\Hom_A(P_M,N))\ot_BBf_i
\cong H^l(g_jC\ot_C\Hom_A(P_M,N)\ot_BBf_i)
\cong H^l(\Hom_A(f_i B\ot_B P_M,g_jC\ot_C N))
\cong \Ext_A^l(f_i M,g_j N)$
for all $1\le i\le m$ and $1\le j\le p$.
By Theorem~\ref{Theorem-HRR-fda} (1), we have
$$\begin{array}{ll}
& \sum\limits_{l\ge 0} (-1)^l\ (\dm \Ext_A^l(M,N))_{ij} \\ [4mm]
= & \sum\limits_{l\ge 0} (-1)^l\ \dim\ g_j\Ext_A^l(M,N)f_i = \sum\limits_{l\ge 0} (-1)^l\ \dim \Ext_A^l(f_i M,g_j N) \\ [4mm]
\stackrel{\rm T}{=} & (\dv f_iM)^T\cdot C_A^{-T}\cdot \dv g_jN = ((\dm M)^T\cdot C_A^{-T}\cdot \dm N)_{ij}
\end{array}$$
for all $1\le i\le m$ and $1\le j\le p$.

{\it Necessity.} Taking $B=k=C$ in Theorem~\ref{Theorem-HRR-Bimod} (1), we obtain Theorem~\ref{Theorem-HRR-fda} (1).

\medskip

Theorem~\ref{Theorem-HRR-Bimod} (2) $\Leftrightarrow$ Theorem~\ref{Theorem-HRR-fda} (2):

{\it Sufficiency.} It is enough to prove
$$\sum\limits_{l\ge 0}\ (-1)^l\ (\dm \Tor^A_l(M,N))_{ij} = (\dm N\cdot C_A^{-1}\cdot \dm M)_{ij}$$
for all $1\le i\le p$ and $1\le j\le m$.
For this, let $P_M$ be any projective resolution of the $B$-$A$-bimodule $M$.
Then $f_jB\ot_B P_M$ is a projective resolution of the right $A$-module $f_jM$.
Note that $f_jB\ot_B-\ot_CCg_i: B\mbox{\rm -Mod-}C \to \Mod k$ is an exact functor.
Thus $f_j\Tor^A_l(M,N) g_i \cong f_jB\ot_BH^l(P_M\ot_AN)\ot_CCg_i \cong H^l(f_jB\ot_BP_M\ot_AN\ot_CCg_i) \cong \Tor^A_l(f_j M,N g_i)$ for all $1\le i\le p$ and  $1\le j\le m$.
By Theorem~\ref{Theorem-HRR-fda} (2), we have
$$\begin{array}{ll}
& \sum\limits_{l\ge 0} (-1)^l\ (\dm \Tor^A_l(M,N))_{ij} \\ [4mm]
= & \sum\limits_{l\ge 0} (-1)^l\ \dim\ f_j\Tor^A_l(M,N)g_i = \sum\limits_{l\ge 0} (-1)^l\ \dim \Tor^A_l(f_jM,Ng_i) \\ [4mm]
\stackrel{\rm T}{=} & \dv Ng_i \cdot C_A^{-1}\cdot \dv f_jM = (\dm N\cdot C_A^{-1}\cdot \dm M)_{ij}
\end{array}$$
for all $1\le i\le p$ and $1\le j\le m$.

{\it Necessity.} Taking $B=k=C$ in Theorem~\ref{Theorem-HRR-Bimod} (2), we get Theorem~\ref{Theorem-HRR-fda} (2).
\end{proof}

\bigskip

\noindent{\bf HRR type formulas on complex level.}
The following theorem gives cohomological, homological, Hochschild cohomological and Hochschild homological four versions of HRR type formulas on complex level for finite dimensional elementary algebras of finite global dimension,
which generalizes at first glance but is essentially equivalent to Theorem~\ref{Theorem-HRR-fda}.

\begin{theorem} \label{Theorem-HRR-Complex}
Let $A$ be a finite dimensional elementary algebra of finite global dimension,
and $\{e_1,\cdots,e_n\}$ a complete set of orthogonal primitive idempotents in $A$.
Then the following four equivalent statements hold:

{\rm (1)} For all bounded complexes $M$ and $N$ of finite dimensional right $A$-modules,
$$\dim(\RHom_A(M,N)) :=\sum\limits_{l\in\Z} (-1)^l\ \dim\Ext_A^l(M,N) = \langle\dv M , \dv N \rangle_A$$
where $\langle -,- \rangle_A : \mathbb{Z}^n\times\mathbb{Z}^n\to\mathbb{Z}, (x,y)\mapsto x^T\cdot C_A^{-T}\cdot y$, is the Ringel form of $A$.

{\rm (2)} For all bounded complex $M$ of finite dimensional right $A$-modules
and bounded complex $N$ of finite dimensional left $A$-modules,
$$\dim(M\ot^L_AN) := \sum\limits_{l\in\Z} (-1)^l\ \dim\Tor^A_l(M,N) = \langle (\dv N)^T ,\dv M \rangle_{A^\op}$$
where $\langle -,- \rangle_{A^\op} : \mathbb{Z}^n\times\mathbb{Z}^n\to\mathbb{Z}, (x,y)\mapsto x^T\cdot C_A^{-1}\cdot y$, is the Ringel form of $A^\op$.

{\rm (3)} For any bounded complex $M$ of finite dimensional $A$-bimodules,
$$\dim(\RHom_{A^e}(A,M)) := \sum\limits_{l\in\Z} (-1)^l\ \dim HH^l(A,M) = \tr(C_A^{-T}\cdot \dm M).$$

{\rm (4)} For any bounded complex $M$ of finite dimensional $A$-bimodules,
$$\dim(A\ot^L_{A^e}M) := \sum\limits_{l\in\Z} (-1)^l\ \dim HH_l(A,M)  = \tr(C_A^{-1}\cdot \dm M).$$
\end{theorem}

\begin{remark} \label{Remark-HRR-Complex-CohFin} {\rm
Since $A$ is a finite dimensional elementary algebra of finite global dimension,
any cohomologically finite dimensional complex of left (resp. right) $A$-modules
is quasi-isomorphic to a bounded complex of finite dimensional projective (injective) left (resp. right) $A$-modules.
Moreover, $A^e$ is also a finite dimensional elementary algebra of finite global dimension.
By Remark~\ref{Remark-DimMatrix-CohomFinDimComplex}, we may freely replace
``bounded complex of finite dimensional left (resp. right) modules'' in Theorem~\ref{Theorem-HRR-Complex} with
``bounded complex of finite dimensional projective (injective) left (resp. right) modules''
or ``cohomologically finite dimensional complex of left (resp. right) modules''.
}\end{remark}

\begin{proof}
By Theorem~\ref{Theorem-HRR-fda}, it is enough to show that Theorem~\ref{Theorem-HRR-Complex} (1) (resp. (2), (3) and (4)) holds if and only if so does Theorem~\ref{Theorem-HRR-fda} (1) (resp. (2), (3) and (4)).

\medskip

Theorem~\ref{Theorem-HRR-Complex} (1) $\Leftrightarrow$ Theorem~\ref{Theorem-HRR-fda} (1):

{\it Sufficiency.}
By Remark~\ref{Remark-HRR-Complex-CohFin}, we may assume that $M$ is a bounded complex of finite dimensional projective right $A$-modules.
By Lemma~\ref{Lemma-DimMatrix-qis} (2), we have
$$\begin{array}{ll}
\sum\limits_{l\in\Z} (-1)^l\ \dim\Ext_A^l(M,N)
& = \dim \Hom_A(M,N) \\
& = \sum\limits_{i,j\in\Z} (-1)^{j-i}\ \dim\Hom_A(M^i,N^j).
\end{array}$$
On the other hand, we have
$$\begin{array}{ll}
\langle\dv M , \dv N \rangle_A
& = \langle \sum\limits_{i\in\Z}(-1)^i\ \dv M^i , \sum\limits_{j\in\Z}(-1)^j\ \dv N^j \rangle_A \\ [5mm]
& =\sum\limits_{i,j\in\Z}(-1)^{i+j}\ \langle \dv M^i , \dv N^j \rangle_A.
\end{array}$$
Now it suffices to prove $\dim\Hom_A(M^i,N^j)=\langle \dv M^i , \dv N^j \rangle_A$ for all $i,j\in\Z$.
This is obvious by Theorem~\ref{Theorem-HRR-fda} (1), since $M^i$ is a finite dimensional projective right $A$-module.

{\it Necessity.} It is clear.

\medskip

Theorem~\ref{Theorem-HRR-Complex} (2) $\Leftrightarrow$ Theorem~\ref{Theorem-HRR-fda} (2):

{\it Sufficiency.} By Remark~\ref{Remark-HRR-Complex-CohFin}, we may assume that $M$ is a bounded complex of finite dimensional projective right $A$-modules.
By Lemma~\ref{Lemma-DimMatrix-qis} (2), we have
$$\sum\limits_{l\in\Z} (-1)^l\ \dim\Tor^A_l(M,N)
= \dim (M \ot_A N) = \sum\limits_{i,j\in\Z} (-1)^{i+j}\ \dim(M^i \ot_A N^j).$$
On the other hand, we have
$$\begin{array}{ll}
\langle (\dv N)^T , \dv M \rangle_{A^\op}
& = \langle \sum\limits_{j\in\Z}(-1)^j\ (\dv N^j)^T ,\sum\limits_{i\in\Z}(-1)^i\ \dv M^i \rangle_{A^\op} \\ [5mm]
& = \sum\limits_{i,j\in\Z}(-1)^{i+j}\ \langle (\dv N^j)^T ,\dv M^i \rangle_{A^\op}.
\end{array}$$
Now it suffices to prove $\dim(M^i \ot_A N^j) = \langle (\dv N^j)^T ,\dv M^i \rangle_{A^\op}$ for all $i,j\in\Z$.
This is obvious by Theorem~\ref{Theorem-HRR-fda} (2), since $M^i$ is a finite dimensional projective right $A$-module.

{\it Necessity.} It is clear.

\medskip

Theorem~\ref{Theorem-HRR-Complex} (3) $\Leftrightarrow$ Theorem~\ref{Theorem-HRR-fda} (3):

{\it Sufficiency.}
By Remark~\ref{Remark-HRR-Complex-CohFin}, we may assume that $M$ is a bounded complex of finite dimensional injective $A$-bimodules.
By Lemma~\ref{Lemma-DimMatrix-qis} (2), we have
$$\sum\limits_{l\in\Z} (-1)^l\ \dim\Ext_{A^e}^l(A,M)
= \dim \Hom_{A^e}(A,M) = \sum\limits_{i\in\Z} (-1)^i\ \dim\Hom_{A^e}(A,M^i).$$
On the other hand, we have
$$\tr(C_A^{-T}\cdot \dm M) = \tr(C_A^{-T}\cdot(\sum\limits_{i\in\Z}(-1)^i\ \dm {M^i}))
=\sum\limits_{i\in\Z}(-1)^i\ \tr(C_A^{-T}\cdot \dm {M^i}).$$
Now it suffices to prove $\dim\Hom_{A^e}(A,M^i)=\tr(C_A^{-T}\cdot \dm {M^i})$ for all $i\in\Z$.
This is obvious by Theorem~\ref{Theorem-HRR-fda} (3), since $M^i$ is a finite dimensional injective $A$-bimodule.

{\it Necessity.} It is clear.

\medskip

Theorem~\ref{Theorem-HRR-Complex} (4) $\Leftrightarrow$ Theorem~\ref{Theorem-HRR-fda} (4):

{\it Sufficiency.}
By Remark~\ref{Remark-HRR-Complex-CohFin}, we may assume that $M$ is a bounded complex of finite dimensional projective $A$-bimodules.
By Lemma~\ref{Lemma-DimMatrix-qis} (2), we have
$$\sum\limits_{l\in\Z} (-1)^l\ \dim\Tor^{A^e}_l(A,M)
= \dim(A\ot_{A^e}M) = \sum\limits_{i\in\Z} (-1)^i\ \dim(A\ot_{A^e}M^i).$$
On the other hand, we have
$$\tr(C_A^{-1}\cdot \dm M) = \tr(C_A^{-1}\cdot (\sum\limits_{i\in\Z}(-1)^i\ \dm {M^i}))
=\sum\limits_{i\in\Z}(-1)^i\ \tr(C_A^{-1}\cdot \dm {M^i}).$$
Now it suffices to prove $\dim(A\ot_{A^e}M^i)=\tr(C_A^{-1}\cdot \dm {M^i})$ for all $i\in\Z$.
This is obvious by Theorem~\ref{Theorem-HRR-fda} (4), since $M^i$ is a finite dimensional projective $A$-bimodule.

{\it Necessity.} It is clear.
\end{proof}

\bigskip

\noindent{\bf HRR type formulas on bimodule complex level.}
The following theorem gives the cohomological and homological HRR type formulas on bimodule complex level for finite dimensional elementary algebras of finite global dimension.
It generalizes at first sight but is essentially equivalent to Theorem~\ref{Theorem-HRR-Complex}.

\begin{theorem} \label{Theorem-HRR-BimodComplex}
Let $A,B$ and $C$ be three finite dimensional elementary algebras, $A$ of finite global dimension,
and $\{e_1,\cdots, e_n\}$, $\{f_1,\cdots,f_m\}$ and $\{g_1,\cdots,g_p\}$ complete sets of orthogonal primitive idempotents of $A$, $B$ and $C$ respectively.
Then the following two equivalent statements hold:

{\rm (1)} For all bounded complex $M$ of finite dimensional $B$-$A$-bimodules
and bounded complex $N$ of finite dimensional $C$-$A$-bimodules,
$$\dm(\RHom_A(M,N)) := \sum\limits_{l\in\Z} (-1)^l\ \dm \Ext_A^l(M,N) = \langle \dm M,\dm N\rangle_A$$
where $\langle -,- \rangle_A : \Z^{n\times m} \times \Z^{n\times p} \to \Z^{m\times p},
(x,y)\mapsto x^T\cdot C_A^{-T}\cdot y$.

{\rm (2)} For all bounded complex $M$ of finite dimensional $B$-$A$-bimodules
and bounded complex $N$ of finite dimensional $A$-$C$-bimodules,
$$\dm(M\ot^L_AN) := \sum\limits_{l\in\Z} (-1)^l\ \dm \Tor^A_l(M,N)  = \langle (\dm N)^T,\dm M \rangle_{A^\op}$$
where $\langle -,- \rangle_{A^\op} : \Z^{n\times p} \times \Z^{n\times m} \to \Z^{p\times m},
(x,y)\mapsto x^T\cdot C_A^{-1}\cdot y$.
\end{theorem}

\begin{proof}
We employ the same proof as Theorem~\ref{Theorem-HRR-Bimod} with merely the following modifications:
Let $P_M$ be any homotopically projective resolution of the bounded complex $M$ of finite dimensional $B$-$A$-bimodules.
Since $f_iB$ is a projective right $B$-module, $f_iB\ot_B P_M$ is a homotopically projective resolution of the bounded complex $f_iB\ot_B M\cong f_iM$ of  finite dimensional right $A$-modules.
\end{proof}

\subsection{Comparisons with HRR type formulas for dg algebras} \label{Subsection-HRR-dga}

The HRR type formulas for dg algebras were given by Shklyarov in \cite{Shk13}.
For any triangular algebra over a field of characteristic zero,
Shklyarov deduced Ringel's formula in Theorem~\ref{Theorem-HRR-fda} (1) from his formula in Theorem~\ref{Theorem-HRR-dga} below (Ref. \cite[5.1]{Shk13}).
In this subsection, we show that, for any finite dimensional elementary algebra of finite global dimension over a field of characteristic zero,
Shklyarov's formula in Theorem~\ref{Theorem-HRR-dga} is just the cohomological HRR formula on complex level in Theorem~\ref{Theorem-HRR-Complex} (1) which is equivalent to Ringel's formula in Theorem~\ref{Theorem-HRR-fda} (1), and Shklyarov's formula in Proposition~\ref{Proposition-HRR-Tensor-dga} below is just the homological HRR formula on complex level in Theorem~\ref{Theorem-HRR-Complex} (2). Two main ingredients in Shklyarov's formulas are Chern character and Shklyarov pairing.

\bigskip

\noindent{\bf Chern character.} Let $A$ be a dg algebra.
Denote by $\mathcal{C}_\mathrm{dg}(A)$ the dg category of all dg right $A$-modules
and $\mathcal{H}A:=H^0(\mathcal{C}_\mathrm{dg}(A))$ the homotopy category of $A$.
A dg right $A$-module $M$ is {\it perfect} if it is a homotopy direct summand of a finitely generated semi-free dg right $A$-module,
i.e., there is a finitely generated semi-free dg right $A$-module $N$
and a pair of degree 0 closed morphisms $\phi:M\to N$ and $\psi:N\to M$ in $\mathcal{C}_\mathrm{dg}(A)$,
or equivalently, morphisms $\phi:M\to N$ and $\psi:N\to M$ in $\mathcal{H}A$,
such that $\psi\phi=\id_M$ in $\mathcal{H}A$.
Denote by $\Per A$ the full dg subcategory of $\mathcal{C}_\mathrm{dg}(A)$ consisting of all perfect dg right $A$-modules.

For each $M\in \mathrm{Per}A$, we have a dg tensor product functor
$T_M := -\ot_kM : \mathrm{Per}k \to \mathrm{Per}A.$
It induces a graded $k$-linear map
$HH_\bullet(T_M) : k=HH_\bullet(k)\cong HH_\bullet(\mathrm{Per}k) \to HH_\bullet(\mathrm{Per}A)\cong HH_\bullet(A)$
by the agreement and functoriality of Hochschild homology \cite{Kel99}.
The {\it Chern character} (or {\it Euler character}, {\it Chern class}, {\it Euler class}) of $M$
is the 0-th Hochschild homology class $\mathrm{ch}(M) := \linebreak HH_0(T_M)(1_k) \in HH_0(A).$

If $M,N \in \mathrm{Per}A$ are homotopy equivalent then $\mathrm{ch}(M)=\mathrm{ch}(N)$ (Ref. \cite[Proposition 3.1]{Shk13}).
If $M \in \mathrm{Per}A$ then $\mathrm{ch}(M[1])=-\mathrm{ch}(M)$.
Moreover, for any triangle $L\to M\to N\to L[1]$ in the homotopy category $\mathcal{H}(\mathrm{Per}A) := H^0(\mathrm{Per}A)$ of $\mathrm{Per}A$,
we have $\mathrm{ch}(M)=\mathrm{ch}(L)+\mathrm{ch}(N)$ (Ref. \cite[Proposition 3.2]{Shk13}).
So we have {\it Chern character map}
$\mathrm{ch} : K_0(\mathcal{H}(\mathrm{Perf}A)) \to HH_0(A), [M]\mapsto \mathrm{ch}(M)$, which is a group homomorphism.
In particular, if $M\in \mathrm{Per}k$ then $\mathrm{ch}(M) = \sum\limits_{l\in\Z}(-1)^l\ \dim H^l(M) \in k$, which is also equal to $\sum\limits_{l\in\Z}(-1)^l\ \dim M^l$ if $M$ is a bounded complex of finite dimensional $k$-vector spaces.

There exists a natural chain isomorphism $(-)^\vee$ from the Hochschild chain complex $C_\bullet(A)$ of $A$
to the Hochschild chain complex $C_\bullet(A^\op)$ of $A^\op$,
and further a graded $k$-linear isomorphism $(-)^\vee : HH_\bullet(A)\to HH_\bullet(A^\op)$.
Moreover, $\mathrm{ch}(\Hom_{\Per A}(M,A))=\mathrm{ch}(M)^\vee$ for all $M\in\Per A$ (Ref. \cite[Proposition 4.5]{Shk13}).

\bigskip

\noindent{\bf Shklyarov pairing.}
Let $A$ be a {\it proper} dg algebra, i.e., a dg algebra of finite dimensional total cohomology.
Then the dg functor $-\ot_{A\ot A^\op}A : \mathrm{Per}(A\ot A^\op) \to \mathrm{Per}k$ induces a chain morphism
$C_\bullet(-\ot_{A\ot A^\op}A) : C_\bullet(\Per(A\ot A^\op)) \rightarrow C_\bullet(\Per k)$
where $C_\bullet(\mathcal{A})$ denotes the Hochschild chain complex of an exact dg category $\mathcal{A}$.
The natural dg functor $\Per A\ot \Per A^\op \to \mathrm{Per}(A\ot A^\op)$ induces a chain morphism
$C_\bullet(\Per A \ot \Per A^\op) \to C_\bullet(\Per(A\ot A^\op)).$
Moreover, we have the K\"{u}nneth morphism $C_\bullet(\Per A)\ot C_\bullet(\Per A^\op) \to C_\bullet(\Per A\ot\Per A^\op)$.
The composition of these three chain morphisms is a chain morphism
$C_\bullet(\Per A)\ot C_\bullet(\Per A^\op) \to C_\bullet(\Per k)$.
Taking cohomology, we obtain a pairing
$\langle -,- \rangle : HH_\bullet(\mathrm{Per}A) \ot HH_\bullet(\mathrm{Per}A^\op) \to k.$
Using the agreement isomorphism $HH_\bullet(A)\cong HH_\bullet(\mathrm{Per}A)$ (Ref. \cite[Theorem 2.4]{Kel99}),
we obtain the {\it Shklyarov pairing}
$$\langle -,- \rangle : HH_\bullet(A) \ot HH_\bullet(A^\op) \to k$$
of $A$ (Ref. \cite[(1.7)]{Shk13}) which is nondegenerate if $A$ is a proper smooth dg algebra \cite[Theorem 1.4]{Shk13}. Recall that a dg algebra $A$ is {\it (homologically) smooth} if $A$ is compact in the derived category ${\cal D}A^e$ of $A^e$, or equivalently, the dg $A$-bimodule $A$ is quasi-isomorphic to a perfect dg $A$-bimodule.

\bigskip

\noindent{\bf HRR type formulas for dg algebras.}
The HRR type formulas for dg algebras are the Shklyarov's formulas below.

\begin{theorem} \label{Theorem-HRR-dga} {\rm (Shklyarov \cite[Theorem 1.2 and Theorem 1.3]{Shk13})}
Let $A$ be a proper dg algebra, and $M,N$ two perfect dg right $A$-modules. Then
$$\mathrm{ch}(\Hom_{\Per A}(M,N)) = \langle \mathrm{ch}(N) , \mathrm{ch}(M)^\vee\rangle$$
where $\langle -,- \rangle$ is the Shklyarov pairing of $A$.
\end{theorem}

\begin{proposition} \label{Proposition-HRR-Tensor-dga} {\rm (Shklyarov \cite[Proposition 4.4]{Shk13})}
Let $A$ be a proper dg algebra, $M$ a perfect dg right $A$-module, and $N$ a perfect dg left $A$-module. Then
$$\mathrm{ch}(M\ot_AN) = \langle \mathrm{ch}(M),\mathrm{ch}(N) \rangle$$
where $\langle -,- \rangle$ is the Shklyarov pairing of $A$.
\end{proposition}

\bigskip

\noindent{\bf Shklyarov pairing for finite dimensional algebras.}
Let $A$ be a finite dimensional elementary algebra of finite global dimension,
and $\{e_1,\cdots,e_n\}$ a complete set of orthogonal primitive idempotents of $A$.
Thanks to Proposition~\ref{Proposition-HH-GloDimFinAlg},
we may assume that $e_1,\cdots,e_n$ and $e_1^\vee:=e_1,\cdots,e_n^\vee:=e_n$
are $k$-bases of $HH_0(A)=A/[A,A]=\bigoplus\limits_{i=1}^nke_i$ and $HH_0(A^\op)=A^\op/[A^\op,A^\op]=\bigoplus\limits_{i=1}^nke_i^\vee$ respectively.
Here, the notations $e_1^\vee,\cdots,e_n^\vee\in A^\op$ are used to distinguish $e_1,\cdots,e_n\in A^\op$ from $e_1,\cdots,e_n\in A$ on one hand,
and are harmonious with the isomorphism $(-)^\vee : HH_0(A)\to HH_0(A^\op)$ on the other hand.
In our situation, the Shklyarov pairing is quite clear.

\begin{proposition} \label{Proposition-Pairing-CartanMatrix}
Let $A$ be a finite dimensional elementary algebra of finite global dimension,
and $\{e_1,\cdots,e_n\}$ a complete set of orthogonal primitive idempotents of $A$.
Then the matrix of the Shklyarov pairing
$\langle -,- \rangle : HH_0(A) \ot HH_0(A^\op) \to k$
under the basis $e_1,\cdots,e_n$ of $HH_0(A)$ and the basis $e_1^\vee,\cdots,e_n^\vee$ of $HH_0(A^\op)$
is the transpose $C_A^T$ of the Cartan matrix $C_A$ of the algebra $A$.
\end{proposition}

\begin{proof}
Note that $\mathrm{ch}(e_iA)=e_i$ for all $1\le i\le n$.
By Theorem~\ref{Theorem-HRR-dga}, we have
$\langle e_i,e_j^\vee\rangle = \langle \mathrm{ch}(e_i A),\mathrm{ch}(e_j A)^\vee\rangle
=\dim\Hom_A(e_j A, e_iA)=(C_A)_{ji}$ for all $1\le i,j\le n$. So the matrix of the Shklyarov pairing $\langle -,- \rangle$ under the bases is $C_A^T$.
\end{proof}

\bigskip

\noindent{\bf Chern characters for finite dimensional algebras.}
For a finite dimensional elementary algebra $A$ of finite global dimension,
up to homotopy equivalences,
$\Per A$ consists of all bounded complexes of finite dimensional projective right $A$-modules.
Moreover, $K_0(\mathcal{H}(\Per A))\cong K_0(\proj A)\cong K_0(\mod A)\cong \Z^n$.

The following proposition relates Chern character with dimension vector,
and provides a concrete formula to calculate Chern character.
In fact, it holds for all perfect dg right $A$-module $M$,
since $\mathrm{ch}(M)$ is invariant under homotopy equivalences of perfect dg right $A$-modules and $\dv M$ is well-defined for any cohomologically finite dimensional complex $M$ of right $A$-modules
and invariant under quasi-isomorphisms of complexes of right $A$-modules (See Remark~\ref{Remark-DimMatrix-CohomFinDimComplex}).

\begin{proposition} \label{Proposition-ChernChar-DimVect}
Let $A$ be a finite dimensional elementary algebra of finite global dimension, $\{e_1,\cdots,e_n\}$ a complete set of orthogonal primitive idempotents of $A$, and $M$ a bounded complex of finite dimensional projective right $A$-modules. Then
$$\mathrm{ch}(M) = (e_1,\cdots,e_n) \cdot C_A^{-1}\cdot \dv M$$
in $HH_0(A)=A/[A,A]=\bigoplus\limits_{i=1}^nke_i$.
\end{proposition}

\begin{proof}
By Proposition~\ref{Proposition-HH-GloDimFinAlg}, we may assume
$\mathrm{ch}(M) = \sum\limits_{j=1}^na_je_j$ with $a_1,\cdots,a_n\in k$.
By Theorem~\ref{Theorem-HRR-dga} and Proposition~\ref{Proposition-Pairing-CartanMatrix}, we have
$(\dv M)_i = \dim\Hom_A(e_iA,M) = \langle \mathrm{ch}(M),\mathrm{ch}(e_iA)^\vee \rangle
= \langle \sum\limits_{j=1}^na_je_j,e_i^\vee \rangle = \sum\limits_{j=1}^n (C_A)_{ij}\cdot a_j$
for all $1\le i\le n$. So $\dv M = C_A\cdot[a_1,\cdots,a_n]$,
i.e., $[a_1,\cdots,a_n] = C_A^{-1} \cdot \dv M.$
Thus $\mathrm{ch}(M) = (e_1,\cdots,e_n) \cdot C_A^{-1}\cdot \dv M$.
\end{proof}

Proposition~\ref{Proposition-ChernChar-DimVect} implies that, under the canonical bases of free modules  of finite ranks,
the Chern character map $\mathrm{ch}: K_0(\mathcal{H}(\Per A))\to HH_0(A)$ is given by $C_A^{-1}$:
$$\xymatrix{
K_0(\mathcal{H}(\Per A)) \ar[r]^{\mathrm{ch}} \ar[d]^\cong & HH_0(A) \ar[d]^\cong & [M] \ar@{|->}[r] \ar@{|->}[d] & \mathrm{ch}(M) \ar@{|->}[d]\\
\Z^n \ar[r]^{\mathrm{ch}} & k^n & \dv M \ar@{|->}[r] & C_A^{-1}\cdot \dv M.}$$

\bigskip

\noindent{\bf Comparisons of the HRR type formulas.}
The cohomological and homological HRR type formulas on complex level in Theorem~\ref{Theorem-HRR-Complex} (1) and (2) are identities in $\Z$,
but Shklyarov's formulas in Theorem~\ref{Theorem-HRR-dga} and Proposition~\ref{Proposition-HRR-Tensor-dga} are identities in $k$.

{\it When both formulas are restricted to a finite dimensional elementary algebra $A$ of finite global dimension
over a field $k$ of characteristic zero,
Shklyarov's formula in Theorem~\ref{Theorem-HRR-dga} and the cohomological HRR type formula on complex level in Theorem~\ref{Theorem-HRR-Complex} (1) coincide:}
Without loss of generality (See Remark~\ref{Remark-HRR-Complex-CohFin}),
let $M$ and $N$ be two bounded complexes of finite dimensional projective right $A$-modules.
Shklyarov's formula in Theorem~\ref{Theorem-HRR-dga} is
$\mathrm{ch}(\Hom_{\mathrm{Per}A}(M,N)) = \langle \mathrm{ch}(N),\mathrm{ch}(M)^\vee \rangle.$
Its left hand side is equal to $\sum\limits_{l\in\Z}(-1)^l\ \dim\Ext^l_A(M,N)$.
By Proposition~\ref{Proposition-Pairing-CartanMatrix} and Proposition~\ref{Proposition-ChernChar-DimVect},
its right hand side
$$\begin{array}{ll}
\langle \mathrm{ch}(N),\mathrm{ch}(M)^\vee \rangle
& = (C_A^{-1} \cdot \dv N)^T \cdot C_A^T \cdot (C_A^{-1} \cdot \dv M) \\ [3mm]
& = (\dv N)^T \cdot C_A^{-1} \cdot \dv M = (\underline{\dim} M)^T \cdot C_A^{-T} \cdot \underline{\dim} N.
\end{array}$$
So $\sum\limits_{l\in\Z}(-1)^l\ \dim\Ext^l_A(M,N) = (\underline{\dim} M)^T \cdot C_A^{-T} \cdot \underline{\dim} N$ in $k$.
Due to $\ch k=0$, $\sum\limits_{l\in\Z}(-1)^l\ \dim\Ext^l_A(M,N) = (\dv M)^T \cdot C_A^{-T} \cdot \dv N$ holds in $\Z$.
This is just the cohomological HRR type formula on complex level in Theorem~\ref{Theorem-HRR-Complex} (1).

{\it When both formulas are restricted to a finite dimensional elementary algebra $A$ of finite global dimension
over a field $k$ of characteristic zero,
Shklyarov's formula in Proposition~\ref{Proposition-HRR-Tensor-dga} and
the homological HRR type formula on complex level in Theorem~\ref{Theorem-HRR-Complex} (2) coincide:}
Without loss of generality (See Remark~\ref{Remark-HRR-Complex-CohFin}),
let $M$ be a bounded complex of finite dimensional projective right $A$-modules
and $N$ a bounded complex of finite dimensional projective left $A$-modules.
Shklyarov's formula in Proposition~\ref{Proposition-HRR-Tensor-dga} is
$\mathrm{ch}(M\ot_AN) = \langle \mathrm{ch}(M),\mathrm{ch}(N) \rangle$.
Its left hand side is equal to $\sum\limits_{l\in\Z}(-1)^l\ \dim \Tor^A_l(M,N)$.
By Proposition~\ref{Proposition-Pairing-CartanMatrix}, Proposition~\ref{Proposition-ChernChar-DimVect} and Lemma~\ref{Lemma-CartanMatrix-AlgebraTensorOp} (2), its right hand side
$$\begin{array}{ll}
\langle \mathrm{ch}(M),\mathrm{ch}(N) \rangle
& = (C_A^{-1}\cdot \dv M)^T \cdot C_A^T \cdot (C_{A^\op}^{-1} \cdot \dv N_{A^\op}) \\ [3mm]
& = (C_A^{-1}\cdot \dv M)^T \cdot C_A^T \cdot (C_A^{-T} \cdot (\dv N)^T) \\ [3mm]
& = (\dv M)^T \cdot C_A^{-T} \cdot (\dv N)^T = \dv N \cdot C_A^{-1} \cdot \dv M.
\end{array}$$
This is just the homological HRR type formula on complex level in Theorem~\ref{Theorem-HRR-Complex} (2).

\section{Lefschetz type formulas for finite dimensional algebras}

In this section, we will give the Lefschetz type formulas for finite dimensional elementary algebras of finite global dimension
and compare them with the Lefschetz type formulas for dg algebras.

\subsection{Trace matrices}

In order to formulate the Lefschetz type formulas for finite dimensional elementary algebras of finite global dimension,
we introduce the trace vector of a module (complex) endomorphism and the trace matrix of a bimodule (complex) endomorphism
which are natural generalizations of the (super) trace of a $k$-vector space (complex) endomorphism.

\bigskip

\noindent{\bf Trace matrices of bimodules.}
Let $A$ be a finite dimensional elementary algebra and $\{e_1,\cdots,e_n\}$ a complete set of orthogonal primitive idempotents of $A$.
The {\it trace vector} of an endomorphism $\phi$ of a finite dimensional right $A$-module $M$ is the column vector $\tv\phi = \tv_M(\phi) := [\tr \phi_1,\cdots,\tr \phi_n] \in k^n$
where $\phi_i\in\End_k(Me_i)$ is the restriction of $\phi$ on $Me_i$ for all $1\le i\le n$.
The {\it trace vector} of an endomorphism $\psi$ of a finite dimensional left $A$-module $N$ is the row vector $\tv\psi = \tv_N(\psi) := (\tr \psi_1,\cdots,\tr \psi_n) \in k^n$
where $\psi_i\in\End_k(e_iN)$ is the restriction of $\psi$ to $e_iN$ for all $1\le i\le n$.

\begin{remark}{\rm
An endomorphism $\phi$ of a finite dimensional right $A$-module $M$ can be viewed naturally as
an endomorphism $\phi$ of a finite dimensional left $A^\op$-module $M$,
but $\tv_{M_A}(\phi)=(\tv_{_{A^\op}M}(\phi))^T$.
}\end{remark}

Let $A$ and $B$ be two finite dimensional elementary algebras, and $\{e_1,\cdots,e_n\}$ and $\{f_1,\cdots,f_m\}$ complete sets of orthogonal primitive idempotents of $A$ and $B$ respectively.
The {\it trace matrix} $\tm\phi=\tm_M(\phi)$ of an endomorphism $\phi$ of a finite dimensional $B$-$A$-bimodule $M$ is the $n\times m$ matrix $(\tr\phi_{ij}) \in k^{n\times m}$ where $\phi_{ij}\in\End_k(f_jMe_i)$ is the restriction of $\phi$ on $f_jMe_i$ for all $1\le i\le n$ and $1\le j\le m$.

\begin{remark}{\rm
An endomorphism $\phi$ of a finite dimensional right $A$-module $M$ can be viewed naturally as
an endomorphism $\phi$ of a finite dimensional $k$-$A$-bimodule $_kM_A$,
and $\tv_{M_A}(\phi) = \tm_{_kM_A}(\phi)$.
An endomorphism $\psi$ of a finite dimensional left $A$-module $N$ can be viewed naturally as
an endomorphism $\psi$ of a finite dimensional $A$-$k$-bimodule $_AN_k$,
and $\tv_{_AN}(\psi)=\tm_{_AN_k}(\psi)$.
In particular, for a $k$-linear operator $\phi$ on a finite dimensional $k$-vector space $M$,
$\tr_M(\phi)=\tm_{_kM_k}(\phi)$. So the trace matrix of an endomorphism of a finite dimensional bimodule
generalizes the trace of a linear operator on a finite dimensional vector space
and the trace vector of an endomorphism of a finite dimensional module.
}\end{remark}

An endomorphism $\phi$ of a finite dimensional $B$-$A$-bimodule $M$ can be viewed naturally as either
an endomorphism $\phi$ of a finite dimensional right $(B^\op\ot A)$-module $M_{B^\op\ot A}$ or
an endomorphism $\phi$ of a finite dimensional left $(B\ot A^\op)$-module $_{B\ot A^\op} M$.
The following lemma distinguishes $\tm_{_BM_A}(\phi)$, $\tv_{M_{B^\op\ot A}}(\phi)$ and $\tv_{_{B\ot A^\op}M}(\phi)$.

\begin{lemma} \label{Lemma-TraceVector-Bimod}
Let $A$ and $B$ be finite dimensional elementary algebras,
$\{e_1,\cdots,e_n\}$ and $\{f_1,\cdots,f_m\}$ complete sets of orthogonal primitive idempotents of $A$ and $B$ respectively,
and $\phi$ an endomorphism of a finite dimensional $B$-$A$-bimodule $M$. Then

{\rm (1)} $\tv_{M_{B^\op\ot A}}(\phi) = [\tv_{f_1M}(\phi_1),\cdots,\tv_{f_mM}(\phi_m)]$,
i.e., the column vectorization of $\tm_M(\phi)$,
where $\phi_j\in\End_A(f_jM)$ is the restriction of $\phi$ on $f_jM$ for all $1\le j\le m$.

{\rm (2)}  $\tv_{_{B\ot A^\op}M}(\phi) = (\tv_{M_{B^\op\ot A}}(\phi))^T$ is the row vectorization of $(\tm_M(\phi))^T$.
\end{lemma}

\begin{proof}
(1) With respect to the complete set $\{f_1\ot e_1,\cdots,f_1\ot e_n,\cdots,f_m\ot e_1,\linebreak \cdots,f_m\ot e_n\}$
of orthogonal primitive idempotent of $B^\op\ot A$, we have
$$\begin{array}{ll}
&\tv_{M_{B^\op\ot A}}(\phi) \\  [2mm]
= & [\tr_{M(f_1\ot e_1)}(\phi_{11}),\cdots,\tr_{M(f_1\ot e_n)}(\phi_{n1}),\cdots,\tr_{M(f_m\ot e_1)}(\phi_{1m}),\cdots,\tr_{M(f_m\ot e_n)}(\phi_{nm})] \\ [2mm]
= & [\tr_{f_1Me_1}(\phi_{11}),\cdots,\tr_{f_1Me_n}(\phi_{n1}),\cdots,\tr_{f_mMe_1}(\phi_{1m}),\cdots,\tr_{f_mMe_n}(\phi_{nm})] \\ [2mm]
= & [\tv_{f_1M}(\phi_1),\cdots,\tv_{f_mM}(\phi_m)].
\end{array}$$

(2) follows from (1).
\end{proof}

Now we observe some properties of trace matrices
whose proofs can be reduced to the corresponding properties of the traces of $k$-linear operators on finite dimensional $k$-vector spaces.

\begin{lemma} \label{Lemma-Property-Trace}
Let $A$ and $B$ be two finite dimensional elementary algebras,
and $\{e_1,\cdots, e_n\}$ and $\{f_1,\cdots,f_m\}$ complete sets of orthogonal primitive idempotents of $A$ and $B$ respectively. Then the following three statements hold:

{\rm (1)} For all endomorphisms $\phi$ and $\psi$ of a finite dimensional $B$-$A$-bimodule $M$, $\tm(\phi+\psi)=\tm\phi+\tm\psi$.

{\rm (2)} For all morphisms $\phi : M\to N$  and $\psi: N\to M$ between finite dimensional $B$-$A$-bimodules $M$ and $N$, $\tm_N(\phi\psi)=\tm_M(\psi\phi).$

{\rm (3)} For any finite dimensional $B$-$A$-bimodule $M$,
$\tm(\id_{_BM_A})=\dm {_BM_A}$ in $k^{n\times m}$.
\end{lemma}

\begin{proof} (1) We have $(\tm(\phi+\psi))_{ij} = \tr((\phi+\psi)_{ij}) = \tr(\phi_{ij}+\psi_{ij})
= \tr\phi_{ij}+\tr\psi_{ij} = (\tm\phi)_{ij}+(\tm\psi)_{ij} = (\tm\phi+\tm\psi)_{ij}$ for all $1\le i\le n$ and $1\le j\le m$.
So $\tm(\phi+\psi)=\tm\phi+\tm\psi$.

(2) We have $(\tm_N(\phi\psi))_{ij} = \tr_{f_jNe_i}((\phi\psi)_{ij}) = \tr_{f_jNe_i}(\phi_{ij}\psi_{ij}) = \linebreak \tr_{f_jMe_i}(\psi_{ij}\phi_{ij}) = \tr_{f_jMe_i}((\psi\phi)_{ij}) = (\tm_M(\psi\phi))_{ij}$ for all $1\le i\le n$ and $1\le j\le m$. So $\tm_N(\phi\psi)=\tm_M(\psi\phi).$

(3) We have $(\tm(\id_{_BM_A}))_{ij} = \tr((\id_{_BM_A})_{ij}) = \tr(\id_{f_jMe_i}) = \dim f_jMe_i = (\dm {_BM_A})_{ij}$ in $k$ for all $1\le i\le n$ and $1\le j\le m$. So $\tm(\id_{_BM_A})=\dm {_BM_A}$ in $k^{n\times m}$.
\end{proof}

We know that any linear operator and its dual operator have the same trace. More general, we have the following lemma:

\begin{lemma} \label{Lemma-TrMatrix-Dual}
Let $A$ and $B$ be finite dimensional elementary algebras,
$\{e_1,\cdots,e_n\}$ and $\{f_1,\cdots,f_m\}$ complete sets of orthogonal primitive idempotents of $A$ and $B$ respectively, $\phi$ an endomorphism of a finite dimensional $B$-$A$-bimodule $M$, and $\phi^*:=\Hom_k(\phi,k)$ the dual endomorphism of $\phi$.
Then $\tm_{M^*}(\phi^*)=(\tm_M(\phi))^T.$
\end{lemma}

\begin{proof} Take a $k$-basis of $f_jMe_i$ for all $1\le i\le n$ and $1\le j\le m$.
They form a $k$-basis of $M$. From the correspondence between dual bases,
we obtain a $k$-linear isomorphism $\xi : e_iM^*f_j \to (f_jMe_i)^*$ which satisfies $\xi\circ(\phi^*)_{ji}=(\phi_{ij})^*\circ\xi$, i.e., the following diagram is commutative:
$$\xymatrix{
e_iM^*f_j \ar[r]^-\xi \ar[d]^-{(\phi^*)_{ji}} & (f_jMe_i)^* \ar[d]^-{(\phi_{ij})^*} \\
e_iM^*f_j \ar[r]^-\xi & (f_jMe_i)^*.
}$$
By Lemma~\ref{Lemma-Property-Trace} (2), we have
$\tr(\phi^*)_{ji} = \tr(\xi^{-1}\circ (\phi_{ij})^*\circ\xi) = \tr(\phi_{ij})^*$. Furthermore,
$(\tm\phi^*)_{ji} = \tr(\phi^*)_{ji} = \tr(\phi_{ij})^* = \tr\phi_{ij} = (\tm \phi)_{ij}$
for all $1\le i\le n$ and $1\le j\le m$. So $\tm\phi^*=(\tm \phi)^T.$
\end{proof}

We know that the trace of the tensor product of two linear operators is the product of their traces. More general, we have the following lemma:

\begin{lemma} \label{Lemma-TrMatrix-TensorProduct}
Let $A$ and $B$ be finite dimensional elementary algebras,
$\{e_1,\cdots,e_n\}$ and $\{f_1,\cdots,f_m\}$ complete sets of orthogonal primitive idempotents of $A$ and $B$ respectively, $\phi$ an endomorphism of a finite dimensional right $A$-module $M$,
and $\psi$ an endomorphism of a finite dimensional left $B$-module $N$.
Then \linebreak $\tm_{N\ot M}(\psi\ot\phi)=\tv_M(\phi) \cdot \tv_N(\psi).$
\end{lemma}

\begin{proof} We have
$(\tm_{N\ot M}(\psi\ot\phi))_{ij} = \tr_{f_j(N\ot M)e_i}((\psi\ot\phi)_{ij}) = \tr_{f_jN\ot Me_i}(\psi_j\ot\phi_i)
= \tr_{Me_i}(\phi_i)\cdot\tr_{f_jN}(\psi_j) = (\tv_M(\phi) \cdot \tv_N(\psi))_{ij}$
for all $1\le i\le n$ and $1\le j\le m$. So $\tm_{N\ot M}(\psi\ot\phi)=\tv_M(\phi) \cdot \tv_N(\psi).$
\end{proof}

\bigskip

\noindent{\bf Trace matrices of complexes.}
Let $\phi$ be a chain endomorphism of a bounded complex $M$ of finite dimensional $k$-vector spaces.
The {\it (super) trace} of $\phi$ is $\tr \phi = \tr_M(\phi) := \sum\limits_{l\in\Z}(-1)^l\ \tr_{M^l}(\phi^l) \in k$.
Let $A$ be a finite dimensional elementary algebra,
and $\{e_1,\cdots,e_n\}$ a complete set of orthogonal primitive idempotents of $A$.
The {\it (super) trace vector} of a chain endomorphism $\phi$ of a bounded complex $M$ of finite dimensional right (resp. left) $A$-modules is the column (resp. row) vector $\tv\phi = \tv_M(\phi) := \sum\limits_{l\in\Z}(-1)^l\ \tv_{M^l}(\phi^l) \in k^n$.
Let $B$ be also a finite dimensional elementary algebra,
and $\{f_1,\cdots, f_m\}$ a complete set of orthogonal primitive idempotents of $B$.
The {\it (super) trace matrix} of a chain endomorphism $\phi$ of a bounded complex $M$ of finite dimensional $B$-$A$-bimodules is $\tm\phi = \tm_M(\phi):= \sum\limits_{l\in\Z}(-1)^l\ \tm_{M^l}(\phi^l) \in k^{n\times m}$.

Recall that the {\it endomorphism category} $\End(\mathcal{C})$ of a category $\mathcal{C}$ is the category
whose objects are all pairs $(C,\phi)$ where $C$ is an object in $\mathcal{C}$ and $\phi\in\End_{\mathcal{C}}(C)$,
whose morphisms $\xi : (C,\phi)\to (C',\phi')$ are all morphisms $\xi\in\Hom_{\mathcal{C}}(C,C')$ satisfying $\xi\phi=\phi'\xi$,
and the composition of two morphisms is just their composition in $\mathcal{C}$ (See, for example, \cite{Len69}).
The following lemma implies that $\tm$ is an additive invariant on the endomorphism category $\End(B\mbox{\rm -mod-}A)$
of the category $B\mbox{\rm -mod-}A$ of finite dimensional $B$-$A$-bimodules.

\begin{lemma} \label{Lemma-TraceMatrix-qis}
Let $A$ and $B$ be two finite dimensional elementary algebras,
and $\{e_1,\cdots,e_n\}$ and $\{f_1,\cdots,f_m\}$ complete sets of orthogonal primitive idempotents of $A$ and $B$ respectively.
Then the following two statements hold:

{\rm (1)} For any short exact sequence $0\to (M',\phi')\xrightarrow{\lambda} (M,\phi)\xrightarrow{\rho} (M'',\phi'')\to 0$
in $\End(B\mbox{\rm -mod-}A)$, i.e., the commutative diagram
$$\xymatrix{
0\ar[r] & M' \ar[r]^{\lambda} \ar[d]^{\phi'} & M \ar[r]^{\rho} \ar[d]^{\phi} & M''\ar[r] \ar[d]^{\phi''} & 0 \\
0\ar[r] & M' \ar[r]^{\lambda} & M \ar[r]^{\rho} & M'' \ar[r] & 0
}$$
in $B\mbox{\rm -mod-}A$ with exact rows, $\tm \phi=\tm \phi' +\tm \phi''$.
So $\tm : K_0(\End(B\mbox{\rm -mod-}A)) \to k^{n\times m}, [(M,\phi)] \mapsto \tm \phi$, is a group homomorphism. In particular, $\tm$ is invariant under the isomorphisms in $\End(B\mbox{\rm -mod-}A)$.

{\rm (2)} For any chain endomorphism $\phi$ of a bounded complex $M$ of finite dimensional $B$-$A$-bimodules,
$$\tm_M(\phi) := \sum\limits_{l\in\Z}(-1)^l\ \tm_{M^l}(\phi^l) = \sum\limits_{l\in\Z}(-1)^l\ \tm_{H^l(M)}(H^l(\phi)).$$
\end{lemma}

\begin{proof}
(1) Acting the exact functor $f_jB\ot_B-\ot_AAe_i : B\mbox{\rm -mod-}A \to \mod k$ on the given short exact sequence in $\End(B\mbox{\rm -mod-}A)$, we obtain a short exact sequence in the endomorphism category $\End(\mbox{\rm mod}k)$ of $\mod k$:
$$\xymatrix{
0\ar[r] & f_jX'e_i \ar[r] \ar[d]^{\phi'_{ij}} & f_jXe_i \ar[r] \ar[d]^{\phi_{ij}} & f_jX''e_i \ar[r] \ar[d]^{\phi''_{ij}} & 0 \\
0\ar[r] & f_jX'e_i \ar[r] & f_jXe_i \ar[r] & f_jX''e_i \ar[r] & 0.
}$$
Thus $\tr\phi_{ij}=\tr\phi'_{ij}+\tr\phi''_{ij}$ for all $1\le i\le n, 1\le j\le m$.
So $\tm\phi=\tm\phi'+\tm\phi''$.

(2) Let $Z^l(M)$ and $B^l(M)$ be the $l$-cocycle and $l$-coboundary of $M$ respectively.
Then we have two short exact sequences
$$\xymatrix{
0\ar[r] & B^l(M) \ar[r] \ar[d]^{B^l(\phi)} & Z^l(M) \ar[r] \ar[d]^{Z^l(\phi)} & H^l(M) \ar[r] \ar[d]^{H^l(\phi)} & 0 \\
0\ar[r] & B^l(M) \ar[r] & Z^l(M) \ar[r] & H^l(M) \ar[r] & 0
}$$
and
$$\xymatrix{
0\ar[r] & Z^l(M) \ar[r] \ar[d]^{Z^l(\phi)} & M^l \ar[r] \ar[d]^{\phi^l} & B^{l+1}(M) \ar[r] \ar[d]^{B^{l+1}(\phi)} & 0 \\
0\ar[r] & Z^l(M) \ar[r] & M^l \ar[r] & B^{l+1}(M) \ar[r] & 0
}$$
in the endomorphism category $\End(B\mbox{\rm -mod-}A)$,
where $B^l(\phi)$ (resp. $Z^l(\phi)$) is the restriction of $\phi^l$ on $B^l(M)$ (resp. $Z^l(M)$) and $H^l(\phi)$ is induced by $\phi^l$ for all $l\in\Z$.
By (1), we have $\tm Z^l(\phi) = \tm B^l(\phi)+\tm H^l(\phi)$
and $\tm \phi^l= \tm Z^l(\phi) + \tm B^{l+1}(\phi)$.
So $\tm \phi^l= \tm B^l(\phi) + \tm B^{l+1}(\phi)+\tm H^l(\phi)$ for all $l\in\Z$.
Thus $\tm \phi := \sum\limits_{l\in\Z}(-1)^l\ \tm \phi^l
= \sum\limits_{l\in\Z}(-1)^l\ \tm H^l(\phi).$
\end{proof}

\begin{remark} \label{Remark-TraceMatrix-CohomFinComplex} {\rm
More general, for a chain endomorphism $\phi$ of a cohomologically finite dimensional complex $M$ of $B$-$A$-bimodules,
we can define its {\it (super) trace matrix}
$\tm \phi = \tm_M(\phi) := \sum\limits_{l\in\Z}(-1)^l\ \tm_{H^l(M)}(H^l(\phi))$.
Due to Lemma~\ref{Lemma-TraceMatrix-qis} (2), this definition extends that for a chain endomorphism $\phi$ of a bounded complex $M$ of finite dimensional $B$-$A$-bimodules.
A morphism $\xi:(M,\phi)\to(N,\psi)$ in the endomorphism category $\End(\mathcal{C}(B\mbox{\rm -Mod-}A))$ of the category
$\mathcal{C}(B\mbox{\rm -Mod-}A)$ of complexes of $B$-$A$-bimodules is a {\it quasi-isomorphism} if $\xi:M\to N$ is a quasi-isomorphism in $\mathcal{C}(B\mbox{\rm -Mod-}A)$. Obviously, trace matrix is invariant under the quasi-isomorphisms in $\End(\mathcal{C}(B\mbox{\rm -Mod-}A))$.
}\end{remark}

\subsection{Lefschetz type formulas for finite dimensional algebras}

In this subsection, using trace vectors and trace matrices,
we will give the Lefschetz type formulas on module, bimodule, module complex and bimodule complex four levels
for finite dimensional elementary algebras of finite global dimension.

\bigskip

\noindent{\bf Lefschetz type formulas on module level.}
The following theorem gives cohomological, homological, Hochschild cohomological and Hochschild homological
four versions of Lefschetz type formulas on module level for finite dimensional elementary algebras of finite global dimension.

\begin{theorem} \label{Theorem-Lefschetz-fda}
Let $A$ be a finite dimensional elementary algebra of finite global dimension,
and $\{e_1,\cdots,e_n\}$ a complete set of orthogonal primitive idempotents of $A$.
Then the following four equivalent statements hold:

{\rm (1)} For all endomorphism $\phi$ of a finite dimensional right $A$-module $M$
and endomorphism $\psi$ of a finite dimensional right $A$-module $N$,
$$\tr(\RHom_A(\phi,\psi)) := \sum\limits_{l\ge 0}(-1)^l\ \tr(\Ext^l_A(\phi,\psi))
= \langle \tv \phi,\tv \psi \rangle_A$$
where $\langle -,- \rangle_A : k^n\times k^n\to k, (x,y)\mapsto x^T\cdot C_A^{-T}\cdot y$.

{\rm (2)} For all endomorphism $\phi$ of a finite dimensional right $A$-module $M$
and endomorphism $\psi$ of a finite dimensional left $A$-module $N$,
$$\tr(\phi\ot^L_A\psi) := \sum\limits_{l\ge 0}(-1)^l\ \tr(\Tor^A_l(\phi,\psi))
= \langle (\tv \psi)^T,\tv \phi \rangle_{A^\op}$$
where $\langle -,- \rangle_{A^\op} : k^n\times k^n\to k, (x,y)\mapsto x^T\cdot C_A^{-1}\cdot y$.

{\rm (3)} For any endomorphism $\phi$ of a finite dimensional $A$-bimodule $M$,
$$\tr(\RHom_{A^e}(A,\phi)) := \sum\limits_{l\ge 0}(-1)^l\ \tr(HH^l(A,\phi)) = \tr(C_A^{-T} \cdot \tm \phi).$$

{\rm (4)} For any endomorphism $\phi$ of a finite dimensional $A$-bimodule $M$,
$$\tr(A\ot^L_{A^e}\phi) := \sum\limits_{l\ge 0}(-1)^l\ \tr(HH_l(A,\phi)) = \tr(C_A^{-1} \cdot\tm \phi).$$
\end{theorem}

\begin{remark}{\rm

(1) In Theorem~\ref{Theorem-Lefschetz-fda},
$\RHom_A(\phi,\psi)$, $\phi\ot^L_A\psi$, $\RHom_{A^e}(A,\phi)$ and $A\ot^L_{A^e}\phi$
are endomorphisms of $\RHom_A(M,N)$, $M\ot^L_AN$, $\RHom_{A^e}(A,M)$ and $A\ot^L_{A^e}M$
in the derived category $\mathcal{D}k$ of $k$ respectively.
Moreover, $\Ext^l_A(\phi,\psi)$, $\Tor^A_l(\phi,\psi)$, $HH^l(A,\phi)$ and $HH_l(A,\phi)$
are $k$-linear operators on $k$-vector spaces $\Ext^l_A(M,N)$, $\Tor^A_l(M,N)$, $HH^l(A,M)$ and $HH_l(A,M)$ respectively.
The leftmost terms of the Lefschetz type formulas on module level in Theorem~\ref{Theorem-Lefschetz-fda}
are just the (super) traces or Hattori-Stallings traces $\tr_{\RHom_A(M,N)}(\RHom_A(\phi,\psi))$, $\tr_{M\ot^L_AN}(\phi\ot^L_A\psi)$, $\tr_{\RHom_{A^e}(A,M)}(\RHom_{A^e}(A,\phi))$
and $\tr_{A\ot^L_{A^e}M}(A\ot^L_{A^e}\phi)$ (See Remark~\ref{Remark-TraceMatrix-CohomFinComplex} and \S \ref{Subsection-Lefschetz-dga}).

(2) The four Lefschetz type formulas in Theorem~\ref{Theorem-Lefschetz-fda} are identities in $k$.
In the case of $A=k$, they are
$\tr(\Hom_k(\phi,\psi)) =\tr\phi\cdot\tr\psi$, $\tr(\phi\ot\psi)=\tr\phi\cdot\tr\psi$,
$\tr(\Hom_k(k,\phi)) =\tr\phi$ and $\tr(k\ot_k\phi)=\tr\phi$
for $k$-linear operators $\phi$ and $\psi$ on finite dimensional $k$-vector spaces.

(3) We do have the identities $\tr(\RHom_A(\phi,\psi)) =  \langle\tv\psi,\tv\phi\rangle_{A^\op}$
and $\tr(\phi\ot^L_A\psi) = \langle \tv\phi,(\tv\psi)^T\rangle_A$.
Nonetheless, they are not so natural due to Theorem~\ref{Theorem-Lefschetz-Bimodule}.

(4) In the case of $\ch k=0$, Theorem~\ref{Theorem-Lefschetz-fda} implies Theorem~\ref{Theorem-HRR-fda}.
For this, it is enough to take all $\phi$ and $\psi$ in Theorem~\ref{Theorem-Lefschetz-fda} to be identity morphisms and apply Lemma~\ref{Lemma-Property-Trace} (3).
}\end{remark}

\begin{proof} First of all, we show that (2) holds.
Denote $S_i:=e_iA/\rad(e_iA), 1\le i\le n$.
Then $\{S_1,\cdots,S_n\}$ is a complete set of representatives of isomorphism classes of simple right $A$-modules,
and $\dv S_1=[1,0,\cdots,0],\cdots,\dv S_n=[0,\cdots,0,1]$.
Since $A$ is a finite dimensional elementary algebra of finite global dimension,
$A^e=A^\op\ot A$ is also a finite dimensional elementary algebra of finite global dimension
and $\{e_i\ot e_j\ |\ 1\le i,j\le n\}$ is a complete set of orthogonal primitive idempotents in $A^e$.
So $\{(e_i\ot e_j)A^e\cong Ae_i\ot e_jA\ |\ 1\le i,j\le n\}$ is a complete set of representatives of isomorphism classes of indecomposable projective $A$-bimodules.
Let $Q_d \rightarrowtail Q_{d-1} \to \cdots \to Q_1 \to Q_0 (\twoheadrightarrow A)$ be a minimal projective resolution of the $A$-bimodule $A$ with $Q_l=\bigoplus\limits_{1\le i,j\le n} (Ae_i \ot e_jA)^{t_{lij}}$ for all $0\le l\le d$.
By \cite[Lemma 1.5]{Hap89}, we have $t_{lij}=\dim\Ext^l_A(S_i,S_j)$ for all $0\le l\le d$ and $1\le i,j\le n$.
Moreover, it follows from Theorem~\ref{Theorem-HRR-fda} (1) that
$$\sum\limits_{l=0}^d (-1)^l\ t_{lij}=(\underline{\dim}S_i)^T \cdot C_A^{-T} \cdot \underline{\dim}S_j.$$
Since $M \ot_A Q_\bullet$ is a projective resolution of the right $A$-module $M$,
we have $\Tor^A_l(M,N)\cong H_l(M \ot_A Q_\bullet \ot_A N)$ for all $0\le l\le d$.
By Lemma~\ref{Lemma-TraceMatrix-qis} (2), we have
$$\begin{array}{ll}
  & \sum\limits_{l\ge 0} (-1)^l\ \tr(\Tor^A_l(\phi,\psi)) \\ [4mm]
\stackrel{\rm L}{=} & \sum\limits_{l=0}^d (-1)^l\
\tr_{M \ot_A Q_l \ot_A N}(\phi\ot \id_{Q_l} \ot\psi) \\ [4mm]
= & \sum\limits_{l=0}^d (-1)^l\sum\limits_{1\le i,j\le n} t_{lij}\cdot
\tr_{M \ot_A (Ae_i \ot e_jA) \ot_A N}(\phi\ot \id_{Ae_i \ot e_jA} \ot\psi) \\ [4mm]
= & \sum\limits_{l=0}^d (-1)^l\sum\limits_{1\le i,j\le n} t_{lij}\cdot \tr_{Me_i \ot e_jN}(\phi_i \ot \psi_j) \\ [4mm]
= & \sum\limits_{l=0}^d (-1)^l\sum\limits_{1\le i,j\le n} t_{lij}\cdot \tr\phi_i\cdot\tr\psi_j \\ [6mm]
= & \sum\limits_{1\le i,j\le n} \tr\phi_i \cdot (\dv S_i)^T \cdot C_A^{-T} \cdot \dv S_j \cdot \tr\psi_j \\ [4mm]
= & (\tr\phi_1,\cdots,\tr\phi_n) \cdot C_A^{-T} \cdot [\tr\psi_1,\cdots,\tr\psi_n] \\ [4mm]
= & (\tv\phi)^T \cdot C_A^{-T} \cdot (\tv\psi)^T  = \tv\psi \cdot C_A^{-1} \cdot \tv\phi.
\end{array}$$

Next we show that the four statements (1), (2), (3) and (4) are equivalent:

(1)$\Rightarrow$(3): By (1), Lemma~\ref{Lemma-Property-Trace} (3), Lemma~\ref{Lemma-DimVector-Bimod} (1), Lemma~\ref{Lemma-CartanMatrix-AlgebraTensorOp} (3) and Lemma~\ref{Lemma-TraceVector-Bimod}, we have
$$\begin{array}{ll}
& \sum\limits_{l\ge 0} (-1)^l\ \tr(HH^l(A,\phi)) \\ [4mm]
= & \sum\limits_{l\ge 0} (-1)^l\ \tr(\Ext_{A^e}^l(A,\phi)) \\ [4mm]
\stackrel{(1)}{=} & (\tv\ \id_{A_{A^e}})^T \cdot C_{A^e}^{-T} \cdot \tv \phi  \\ [4mm]
\stackrel{\rm L}{=} & (\dv A_{A^e})^T \cdot C_{A^e}^{-T} \cdot \tv \phi \\ [4mm]
\stackrel{\rm 3L}{=} & \begin{pmatrix} \dv e_1A \\ \vdots \\ \dv e_nA \end{pmatrix}^T \cdot (C_A^{-1} \ot C_A^{-T}) \cdot
\begin{pmatrix} \tv_{e_1M}(\phi_1)\\ \vdots \\ \tv_{e_nM}(\phi_n) \end{pmatrix} \\ [10mm]
= & \begin{pmatrix} \dv e_1A\\ \vdots \\ \dv e_nA \end{pmatrix}^T \cdot
\begin{pmatrix} (C_A^{-1})_{11}\cdot C_A^{-T} & \cdots & (C_A^{-1})_{1n}\cdot C_A^{-T}\\
\vdots & \ddots & \vdots\\ (C_A^{-1})_{n1}\cdot C_A^{-T} & \cdots & (C_A^{-1})_{nn}\cdot C_A^{-T} \end{pmatrix} \cdot
\begin{pmatrix} \tv_{e_1M}(\phi_1) \\ \vdots \\ \tv_{e_nM}(\phi_n) \end{pmatrix} \\ [10mm]
= & \sum\limits_{1\le i,j\le n} (\dv e_iA)^T\cdot (C_A^{-1})_{ij}\cdot C_A^{-T} \cdot \tv_{e_jM}(\phi_j) \\ [4mm]
\stackrel{(1)}{=} & \sum\limits_{1\le i,j\le n} (C_A^{-1})_{ij} \cdot \tr(\Hom_A(e_iA,\phi_j)) \\ [5mm]
= & \sum\limits_{1\le i,j\le n} (C_A^{-1})_{ij} \cdot (\tm\phi)_{ij} \\ [5mm]
= & \tr(C_A^{-T} \cdot \tm\phi).
\end{array}$$

(3)$\Rightarrow$(4):
Note that
$$\begin{array}{ll}
\tr(HH_l(A,\phi)) & = \tr(H_l(A\ot_{A^e}^L\phi)) = \tr(H^l((A\ot_{A^e}^L\phi)^*)) \\ [2mm]
& = \tr(H^l(\RHom_{A^e}(A,\phi^*))) = \tr(HH^l(A,\phi^*))
\end{array}$$
for all $l\in\Z$. By (3) and Lemma~\ref{Lemma-TrMatrix-Dual}, we have
$$\begin{array}{ll}
\sum\limits_{l\ge 0} (-1)^l\ \tr(HH_l(A,\phi))& = \sum\limits_{l\ge 0} (-1)^l\ \tr(HH^l(A,\phi^*)) \stackrel{(3)}{=} \tr(C_A^{-T}\cdot\tm\phi^*) \\ [4mm]
& \stackrel{\rm L}{=} \tr(C_A^{-T}\cdot(\tm\phi)^T) = \tr(\tm\phi \cdot C_A^{-1}) = \tr(C_A^{-1} \cdot \tm\phi).
\end{array}$$

(4)$\Rightarrow$(2):
Note that
$$\begin{array}{ll}
\tr(\Tor^A_l(\phi,\psi)) & = \tr(H_l(\phi\ot^L_A\psi)) = \tr(H_l(A\ot^L_{A^e}(\psi\ot\phi))) \\ [2mm]
& = \tr(HH_l(A,\psi\ot\phi))
\end{array}$$
for all $l\in\Z$. By (4) and Lemma~\ref{Lemma-TrMatrix-TensorProduct}, we have
$$\begin{array}{ll}
& \sum\limits_{l\ge 0} (-1)^l\ \tr(\Tor^A_l(\phi,\psi)) \\ [4mm]
= & \sum\limits_{l\ge 0} (-1)^l\ \tr(HH_l(A,\psi\ot\phi))
\stackrel{(4)}{=} \tr(C_A^{-1}\cdot \tm({\psi\ot\phi}))
\stackrel{\rm L}{=} \tr(C_A^{-1} \cdot \tv \phi \cdot \tv \psi) \\ [4mm]
= & \tr(\tv \psi \cdot C_A^{-1} \cdot \tv \phi)
= \tv \psi \cdot C_A^{-1} \cdot \tv \phi = \langle (\tv \psi)^T,\tv \phi \rangle_{A^\op}.
\end{array}$$

(2)$\Rightarrow$(1):
Note that
$$\begin{array}{ll}
\tr(\Ext_A^l(\phi,\psi)) & = \tr(H^l(\RHom_A(\phi,\psi))) = \tr(H^l((\phi\ot^L_A\psi^*)^*)) \\ [2mm]
& = \tr(H_l(\phi\ot^L_A\psi^*)) = \tr(\Tor^A_l(\phi,\psi^*))
\end{array}$$
for all $l\in\Z$.
By (2) and Lemma~\ref{Lemma-TrMatrix-Dual},
we have
$$\begin{array}{ll}
\sum\limits_{l\ge 0} (-1)^l\ \tr(\Ext_A^l(\phi,\psi))
& = \sum\limits_{l\ge 0} (-1)^l\ \tr(\Tor^A_l(\phi,\psi^*))
\stackrel{(2)}{=} \tv\psi^* \cdot C_A^{-1}\cdot \tv\phi \\ [4mm]
& \stackrel{\rm L}{=} (\tv\psi)^T \cdot C_A^{-1}\cdot \tv\phi
= (\tv \phi)^T\cdot C_A^{-T}\cdot \tv \psi.
\end{array}$$

Therefore, the four statements (1), (2), (3) and (4) are equivalent.
\end{proof}

\bigskip

\noindent{\bf Lefschetz type formulas on bimodule level.}
The following theorem gives the cohomological and homological Lefschetz type formulas on bimodule level for finite dimensional elementary algebras of finite global dimension,
which generalizes at first sight but is essentially equivalent to Theorem~\ref{Theorem-Lefschetz-fda}.

\begin{theorem} \label{Theorem-Lefschetz-Bimodule}
Let $A,B$ and $C$ be three finite dimensional elementary algebras, $A$ of finite global dimension,
and $\{e_1,\cdots, e_n\}$, $\{f_1,\cdots,f_m\}$ and $\{g_1,\cdots,g_p\}$ complete sets of orthogonal primitive idempotents of $A$, $B$ and $C$ respectively.
Then the following two equivalent statements hold:

{\rm (1)} For all endomorphism $\phi$ of finite dimensional $B$-$A$-bimodule $M$ and
 endomorphism $\psi$ of finite dimensional $C$-$A$-bimodule $N$,
$$\tm(\RHom_A(\phi,\psi)) := \sum\limits_{l\ge 0}(-1)^l\ \tm(\Ext^l_A(\phi,\psi))
= \langle \tm \phi,\tm \psi \rangle_A$$
where $\langle -,- \rangle_A : k^{n\times m} \times k^{n\times p} \to k^{m\times p},
(x,y)\mapsto x^T\cdot C_A^{-T}\cdot y$.

{\rm (2)} For all endomorphism $\phi$ of finite dimensional $B$-$A$-bimodule $M$ and
endomorphism $\psi$ of finite dimensional $A$-$C$-bimodule $N$,
$$\tm(\phi\ot^L_A\psi) := \sum\limits_{l\ge 0}(-1)^l\ \tm(\Tor^A_l(\phi,\psi))
= \langle (\tm \psi)^T,\tm\phi \rangle_{A^\op}$$
where $\langle -,- \rangle_{A^\op} : k^{n\times p} \times k^{n\times m} \to k^{p\times m},
(x,y)\mapsto x^T\cdot C_A^{-1}\cdot y$.
\end{theorem}

\begin{proof}
By Theorem~\ref{Theorem-Lefschetz-fda}, it is enough to show that Theorem~\ref{Theorem-Lefschetz-Bimodule} (1) (resp. (2)) holds if and only if
so does Theorem~\ref{Theorem-Lefschetz-fda} (1) (resp. (2)).
We just prove that Theorem~\ref{Theorem-Lefschetz-Bimodule} (1) holds if and only if so does Theorem~\ref{Theorem-Lefschetz-fda} (1).
Similar for the other.

{\it Sufficiency.} It suffices to prove
$$\sum\limits_{l\ge 0}(-1)^l\ (\tm(\Ext^l_A(\phi,\psi)))_{ij}
= ((\tm \phi)^T \cdot C_A^{-T} \cdot \tm \psi)_{ij}$$
for all $1\le i\le m$ and $1\le j\le p$.
For this, let $P_M (\stackrel{\varepsilon}{\twoheadrightarrow} M)$ be any projective resolution of the $B$-$A$-bimodule $M$.
Then the endomorphism $\phi$ of the $B$-$A$-bimodule $M$ can be lifted to a chain endomorphism $\tilde{\phi}$ of $P_M$ such that the following diagram is commutative in the category of complexes of $B$-$A$-bimodules:
$$\xymatrix{
P_M \ar@{->>}[r]^\varepsilon \ar[d]^-{\tilde{\phi}} & M \ar[d]^{\phi} \\
P_M \ar@{->>}[r]^\varepsilon & M. }$$
Since $f_iB$ is a projective right $B$-module, $f_iB\ot_B P_M$ is a projective resolution of the right $A$-module $f_iM$,
and the chain endomorphism $f_iB\ot_B \tilde{\phi}$ of $f_iB\ot_B P_M$ is a lift of the endomorphism $\phi_i$ of $f_iM$ such that the following diagram is commutative in the category of complexes of right $A$-modules:
$$\xymatrix{
f_iB\ot_B P_M \ar@{->>}[r] \ar[d]^-{f_iB\ot_B\tilde{\phi}} & f_iM \ar[d]^{\phi_i} \\
f_iB\ot_B P_M \ar@{->>}[r] & f_iM. }$$
Since $g_jC\ot_C-\ot_BBf_i: C\mbox{\rm -Mod-}B \to \Mod k$ is an exact functor, we have a series of isomorphisms
$$\begin{array}{ll}
& (g_j\Ext_A^l(M,N)f_i,\ (\Ext^l_A(\phi,\psi))_{ij}) \\ [2mm]
\cong & (g_jC\ot_CH^l(\Hom_A(P_M,N))\ot_BBf_i,\ g_jC\ot_CH^l(\Hom_A(\tilde{\phi},\psi))\ot_BBf_i) \\ [2mm]
\cong & (H^l(g_jC\ot_C\Hom_A(P_M,N)\ot_BBf_i),\ H^l(g_jC\ot_C\Hom_A(\tilde{\phi},\psi)\ot_BBf_i)) \\ [2mm]
\cong & (H^l(\Hom_A(f_i B\ot_B P_M,g_jC\ot_C N)),\ H^l(\Hom_A(f_i B\ot_B \tilde{\phi},g_jC\ot_C \psi))) \\ [2mm]
\cong & (\Ext_A^l(f_i M,g_j N),\ \Ext_A^l(\phi_i,\psi_j))
\end{array}$$
in the endomorphism category $\End(\mod k)$ of $\mod k$ for all $l\in\mathbb{N}, 1\le i\le m$ and $1\le j\le p$.
By Lemma~\ref{Lemma-TraceMatrix-qis} (1) and Theorem~\ref{Theorem-Lefschetz-fda} (1), we have
$$\begin{array}{ll}
& \sum\limits_{l\ge 0}(-1)^l\ (\tm(\Ext^l_A(\phi,\psi)))_{ij} \\ [4mm]
= & \sum\limits_{l\ge 0}(-1)^l\ \tr_{g_j\Ext_A^l(M,N)f_i} ((\Ext^l_A(\phi,\psi))_{ij}) \\ [4mm]
\stackrel{\rm L}{=} & \sum\limits_{l\ge 0}(-1)^l\ \tr_{g_jC\ot_CH^l(\Hom_A(P_M,N))\ot_BBf_i} (g_jC\ot_CH^l(\Hom_A(\tilde{\phi},\psi))\ot_BBf_i) \\ [4mm]
\stackrel{\rm L}{=} & \sum\limits_{l\ge 0}(-1)^l\ \tr_{H^l(g_jC\ot_C\Hom_A(P_M,N)\ot_BBf_i)} (H^l(g_jC\ot_C\Hom_A(\tilde{\phi},\psi)\ot_BBf_i)) \\ [4mm]
\stackrel{\rm L}{=} & \sum\limits_{l\ge 0}(-1)^l\ \tr_{H^l(\Hom_A(f_i B\ot_B P_M,g_jC\ot_C N))} (H^l(\Hom_A(f_i B\ot_B \tilde{\phi},g_jC\ot_C \psi))) \\ [4mm]
\stackrel{\rm L}{=} & \sum\limits_{l\ge 0}(-1)^l\ \tr (\Ext_A^l(\phi_i,\psi_j)) \\ [4mm]
\stackrel{\rm T}{=} & (\tv \phi_i)^T\cdot C_A^{-T}\cdot \tv \psi_j = ((\tm \phi)^T \cdot C_A^{-T} \cdot \tm \psi)_{ij}
\end{array}$$
for all $1\le i\le m$ and $1\le j\le p$.

{\it Necessity.} Take $B=k=C$ in Theorem~\ref{Theorem-Lefschetz-Bimodule} (1),
we obtain Theorem~\ref{Theorem-Lefschetz-fda} (1).
\end{proof}

\bigskip

\noindent{\bf Lefschetz type formulas on complex level.}
The following theorem gives cohomological, homological, Hochschild cohomological and Hochschild homological four versions of Lefschetz type formulas on complex level for finite dimensional elementary algebras of finite global dimension,
which generalizes at first glance but is essentially equivalent to Theorem~\ref{Theorem-Lefschetz-fda}.

\begin{theorem} \label{Theorem-Lefschetz-Complex}
Let $A$ be a finite dimensional elementary algebra of finite global dimension,
and $\{e_1,\cdots,e_n\}$ a complete set of orthogonal primitive idempotents of $A$.
Then the following four equivalent statements hold:

{\rm (1)} For all chain endomorphism $\phi$ of a bounded complex $M$ of finite dimensional right $A$-modules
and chain endomorphism $\psi$ of a bounded complex $N$ of finite dimensional right $A$-modules,
$$\tr(\RHom_A(\phi,\psi)) := \sum\limits_{l\in\Z}(-1)^l\ \tr(\Ext^l_A(\phi,\psi))
= \langle \tv \phi,\tv \psi \rangle_A$$
where $\langle -,- \rangle_A : k^n\times k^n\to k, (x,y)\mapsto x^T\cdot C_A^{-T}\cdot y$.

{\rm (2)} For all chain endomorphism $\phi$ of a bounded complex $M$ of finite dimensional right $A$-modules
and chain endomorphism $\psi$ of a bounded complex $N$ of finite dimensional left $A$-modules,
$$\tr(\phi\ot^L_A\psi) := \sum\limits_{l\in\Z}(-1)^l\ \tr(\Tor^A_l(\phi,\psi))
= \langle (\tv \psi)^T,\tv \phi \rangle_{A^\op}$$
where $\langle -,- \rangle_{A^\op} : k^n\times k^n\to k, (x,y)\mapsto x^T\cdot C_A^{-1}\cdot y$.

{\rm (3)} For any chain endomorphism $\phi$ of a bounded complex $M$ of finite dimensional $A$-bimodules,
$$\tr(\RHom_{A^e}(A,\phi)) := \sum\limits_{l\in\Z}(-1)^l\ \tr(HH^l(A,\phi))
= \tr(C_A^{-T} \cdot \tm \phi).$$

{\rm (4)} For any chain endomorphism $\phi$ of a bounded complex $M$ of finite dimensional $A$-bimodules,
$$\tr(A\ot^L_{A^e}\phi) := \sum\limits_{l\in\Z}(-1)^l\ \tr(HH_l(A,\phi)) = \tr(C_A^{-1} \cdot\tm \phi).$$
\end{theorem}

\begin{remark} \label{Remark-Lefschetz-Complex-CohFin} {\rm
By Remark~\ref{Remark-HRR-Complex-CohFin} and Remark~\ref{Remark-TraceMatrix-CohomFinComplex}, we may freely replace
``bounded complex of finite dimensional left (resp. right) modules'' in Theorem~\ref{Theorem-Lefschetz-Complex} with
``bounded complex of finite dimensional projective (injective) left (resp. right) modules''
or ``cohomologically finite dimensional complex of left (resp. right) modules''.
}\end{remark}

\begin{proof}
By Theorem~\ref{Theorem-Lefschetz-fda}, it is enough to show that Theorem~\ref{Theorem-Lefschetz-Complex} (1) (resp. (2), (3) and (4)) holds if and only if
so does Theorem~\ref{Theorem-Lefschetz-fda} (1) (resp. (2), (3) and (4)).
We just prove that Theorem~\ref{Theorem-Lefschetz-Complex} (1) holds if and only if so does Theorem~\ref{Theorem-Lefschetz-fda} (1).
Similar for the others.

{\it Sufficiency.} Since $A$ is a finite dimensional elementary algebra of finite global dimension,
for any bounded complex $M$ of finite dimensional right $A$-modules,
there exist a bounded complex $P_M$ of finite dimensional projective right $A$-modules
and a quasi-isomorphism $\varepsilon : P_M\to M$.
Moreover, the chain endomorphism $\phi$ of $M$ can be lifted to
a chain endomorphism $\tilde{\phi}$ of $P_M$ such that $\phi\varepsilon=\varepsilon\tilde{\phi}$.
$$\xymatrix{
P_M \ar[r]^\varepsilon \ar[d]_{\tilde{\phi}} & M \ar[d]^\phi & H^l(P_M) \ar[r]^{H^l(\varepsilon)} \ar[d]_{H^l(\tilde{\phi})} & H^l(M) \ar[d]^{H^l(\phi)} \\
P_M \ar[r]^\varepsilon & M & H^l(P_M) \ar[r]^{H^l(\varepsilon)} & H^l(M)
}$$
Then $(H^l(P_M),H^l(\tilde{\phi}))\cong (H^l(M),H^l(\phi))$ in the endomorphism category \linebreak $\End(\mod k)$ for all $l\in\Z$.
By Lemma~\ref{Lemma-TraceMatrix-qis} (1), we have $\tv(H^l(\phi)) = \tv(H^l(\tilde{\phi}))$ for all $l\in\Z$.
It follows from Lemma~\ref{Lemma-TraceMatrix-qis} (2) that $\tv\phi=\tv\tilde{\phi}$.
Thus we may assume that $M$ itself is a bounded complex of finite dimensional projective right $A$-modules.
By Lemma~\ref{Lemma-TraceMatrix-qis} (2) again, we have
$$\sum\limits_{l\in\Z} (-1)^l\ \tr(\Ext_A^l(\phi,\psi))
=\tr(\Hom_A(\phi,\psi)) = \sum\limits_{i,j\in\Z} (-1)^{j-i}\ \tr(\Hom_A(\phi^i,\psi^j)).$$
On the other hand, we have
$$\langle \tv\phi,\tv\psi \rangle_A
= \langle \sum\limits_{i\in\Z}(-1)^i\ \tv \phi^i , \sum\limits_{j\in\Z}(-1)^j\ \tv \psi^j \rangle_A
= \sum\limits_{i,j\in\Z}(-1)^{i+j}\ \langle \tv\phi^i , \tv\psi^j \rangle_A.$$
Now it suffices to show $\tr(\Hom_A(\phi^i,\psi^j))= \langle \tv\phi^i,\tv\psi^j \rangle_A$ for all $i,j\in\Z$.
This is obvious by Theorem~\ref{Theorem-Lefschetz-fda} (1), since $M^i$ is a finite dimensional projective right $A$-module.

{\it Necessity.} It is clear.
\end{proof}

\bigskip

\noindent{\bf Lefschetz type formulas on bimodule complex level.}
The following result gives the cohomological and homological Lefschetz type formulas on bimodule complex level for finite dimensional elementary algebras of finite global dimension,
which generalizes at first sight but is essentially equivalent to Theorem~\ref{Theorem-Lefschetz-Complex}.

\begin{theorem} \label{Theorem-Lefschetz-BiModComplex}
Let $A,B$ and $C$ be three finite dimensional elementary algebras, $A$ of finite global dimension,
and $\{e_1,\cdots, e_n\}$, $\{f_1,\cdots,f_m\}$ and $\{g_1,\cdots,g_p\}$ complete sets of orthogonal primitive idempotents of $A$, $B$ and $C$ respectively.
Then the following two equivalent statements hold:

{\rm (1)} For all chain endomorphism $\phi$ of a bounded complex $M$ of finite dimensional $B$-$A$-bimodules
and chain endomorphism $\psi$ of a bounded complex $N$ of finite dimensional $C$-$A$-bimodules,
$$\tm(\RHom_A(\phi,\psi)) := \sum\limits_{l\in\Z}(-1)^l\ \tm(\Ext^l_A(\phi,\psi))
= \langle \tm \phi,\tm \psi \rangle_A$$
where $\langle -,- \rangle_A : k^{n\times m} \times k^{n\times p} \to k^{m\times p},
(x,y)\mapsto x^T\cdot C_A^{-T}\cdot y$.

{\rm (2)} For all chain endomorphism $\phi$ of a bounded complex $M$ of finite dimensional $B$-$A$-bimodules
and chain endomorphism $\psi$ of a bounded complex $N$ of finite dimensional $A$-$C$-bimodules,
$$\tm(\phi\ot^L_A\psi) := \sum\limits_{l\in\Z}(-1)^l\ \tm(\Tor^A_l(\phi,\psi))
= \langle (\tm \psi)^T,\tm \phi \rangle_{A^\op}$$
where $\langle -,- \rangle_{A^\op} : k^{n\times p} \times k^{n\times m} \to k^{p \times m},
(x,y)\mapsto x^T\cdot C_A^{-1}\cdot y$.
\end{theorem}

\begin{proof}
We may employ the same proof as Theorem~\ref{Theorem-Lefschetz-Bimodule} with merely the following modifications:
Let $P_M$ be any homotopically projective resolution of the bounded complex $M$ of finite dimensional $B$-$A$-bimodules.
Since $f_iB$ is a projective right $B$-module, $f_iB\ot_B P_M$ is a homotopically projective resolution of the bounded complex $f_iM$ of  finite dimensional right $A$-modules.
\end{proof}

\subsection{Comparisons with Lefschetz type formulas for dg algebras} \label{Subsection-Lefschetz-dga}

The Lefschetz type formulas for dg algebras were given by Petit in \cite[Proposition 5.5 and Theorem 5.6]{Pet13}.
In this subsection, we will show that the homological Lefschetz type formula on complex level in Theorem~\ref{Theorem-Lefschetz-Complex} (2)
(resp. the Hochschild homological Lefschetz type formula on complex level in Theorem~\ref{Theorem-Lefschetz-Complex} (4))
is just Petit's formula in \cite[Theorem 5.6]{Pet13} (resp. \cite[Proposition 5.5]{Pet13})
restricted to finite dimensional elementary algebras of finite global dimension.
Two main ingredients in Petit's formulas are Hattori-Stallings trace (or Hochschild class) and Shklyarov pairing.
The latter has been introduced already in last section.

\bigskip

\noindent{\bf Hattori-Stallings traces.}
Hattori-Stallings trace is a generalization of the trace of a linear operator.
Let $A$ be an algebra and $P$ a finitely generated projective right $A$-module.
The {\it Hattori-Stallings trace map} of $P$ is the $k$-linear map
$\tr_P : \End_A(P) \cong P \ot_A \Hom_A(P,A) \cong A\ot_{A^e}(\Hom_A(P,A)\ot P) \xrightarrow{\id_A\ot \mathrm{ev}_P} A\ot_{A^e}A \cong HH_0(A)$
where the $A$-bimodule morphism $\mathrm{ev}_P: \Hom_A(P,A)\ot P\to A,\ \xi\ot p \mapsto \xi(p)$, is the evaluation map.
For any $\phi\in \End_A(P)$, the {\it Hattori-Stallings trace} of $\phi$ is $\tr_P(\phi) \in HH_0(A)$.
We refer to \cite{Len69} for Hattori-Stallings trace theory which played crucial roles
in the proofs of the {\it strong no loop conjecture} for finite dimensional elementary algebras \cite{IguLiuPaq11,Han15}.

More general, let $A$ be a dg algebra. Denote by $\per A$ the {\it perfect derived category} of $A$,
i.e., the full triangulated subcategory of the derived category $\mathcal{D}A$ of $A$ consisting of all compact objects,
or equivalently, the smallest thick triangulated subcategory of $\mathcal{D}A$ containing $A$.
Moreover, $\per A$ is triangle equivalent to the homotopy category $\mathcal{H}(\Per A) := H^0(\Per A)$ of the dg category $\Per A$ of perfect dg $A$-modules.
The {\it derived Hattori-Stallings trace morphism} of an object $P\in\per A$ is the morphism
$\Tr_P : \RHom_A(P,P) \cong P \ot^L_A \RHom_A(P,A) \cong A\ot^L_{A^e}(\RHom_A(P,A)\ot P) \xrightarrow{A\ot^L_{A^e}\mathrm{ev}_P} A\ot^L_{A^e}A$ in the derived category $\mathcal{D}k$ of $k$,
where the evaluation morphism $\mathrm{ev}_P: \RHom_A(P,A)\ot P \to A$ is the morphism in $\mathcal{D}A^e$ corresponding to the identity morphism
$\id_{\RHom_A(P,A)}$ under the adjoint isomorphism
$$\Hom_{\mathcal{D}A^e}(\RHom_A(P,A)\ot P,A) \cong \Hom_{\mathcal{D}A^\op}(\RHom_A(P,A),\RHom_A(P,A)),$$
i.e., the counit of the adjoint pair $-\ot P: \mathcal{D}A^\op \rightleftarrows \mathcal{D}A^e :\RHom_A(P,-)$.
Taking 0-th cohomology, we obtain the {\it Hattori-Stallings trace map} or {\it Hochschild class map}
$\tr_P := H^0(\Tr_P) : \End_{\per A}(P) \to HH_0(A)$.
For any $\phi\in \End_{\per A}(P)$, the {\it Hattori-Stallings trace} or {\it Hochschild class} of $\phi$ is $\tr_P(\phi)\in HH_0(A)$.

\bigskip

\noindent{\bf Lefschetz type formulas for dg algebras.}
The Lefschetz type formulas for dg algebras are Petit's formulas below.

\begin{theorem} \label{Theorem-Lefschetz-dga} {\rm (Petit \cite[Theorem 5.6]{Pet13})}
Let $A$ be a proper smooth dg algebra, $M\in \per A$, $\phi\in\End_{\per A}(M)$, $N\in \per A^\op$ and $\psi\in\End_{\per A^\op}(N)$. Then
$$\tr_{M\ot^L_AN}(\phi\ot^L_A\psi) = \langle \tr_M(\phi),\tr_N(\psi) \rangle$$
where $\langle -,- \rangle$ is the Shklyarov pairing of $A$.
\end{theorem}

\begin{proposition} \label{Proposition-Lefschetz-HochHom-dga} {\rm (Petit \cite[Proposition 5.5]{Pet13})}
Let $A$ be a proper smooth dg algebra, $M\in\per A^e$ and $\phi\in\End_{\per A^e}(M)$. Then
$$\tr_{A\ot^L_{A^e}M}(A\ot^L_{A^e}\phi)=
\langle \tr_{A_{A^e}}(\id_{A_{A^e}}),\tr_{_{A^e}M}(\phi) \rangle$$
where $\langle -,- \rangle$ is the Shklyarov pairing of $A^e$.
\end{proposition}

\bigskip

\noindent{\bf Hattori-Stallings traces for finite dimensional algebras.}
For any finite dimensional elementary algebra $A$ of finite global dimension,
its perfect derived category $\per A$ is triangle equivalent to
the homotopy category $\mathcal{H}^b(\proj A)$ of the category of bounded complexes of finite dimensional projective right $A$-modules,
or the homotopy category $\mathcal{H}^b(\inj A)$ of the category of bounded complexes of finite dimensional injective right $A$-modules,
or the bounded derived category $\mathcal{D}^b(\mod A)$ of the category $\mod A$ of finite dimensional right $A$-modules,
or the full triangulated subcategory $\mathcal{D}^f(A)$ of the derived category $\mathcal{D}A$ of $A$ consisting of cohomologically finite dimensional complexes of right $A$-modules.

The following result relates Hattori-Stallings trace with trace vector,
and provides a concrete formula to calculate Hattori-Stallings trace.

\begin{proposition} \label{Proposition-HochClass-TrCartan}
Let $A$ be a finite dimensional elementary algebra of finite global dimension,
$\{e_1,\cdots,e_n\}$ a complete set of orthogonal primitive idempotents in $A$,
and $\phi$ a chain endomorphism of a bounded complex $M$ of finite dimensional right $A$-modules.
Then the Hattori-Stallings trace of $\phi$
$$\tr_M(\phi) =(e_1,\cdots,e_n)\cdot C_A^{-1}\cdot \tv_M(\phi)$$
in $HH_0(A)=A/[A,A]=\bigoplus\limits_{i=1}^nke_i$.
\end{proposition}

\begin{proof}
It follows from Proposition~\ref{Proposition-Pairing-CartanMatrix} that the matrix of the Shklyarov pairing $\langle -,- \rangle$ of $A$ under the $k$-basis $e_1,\cdots,e_n$ of $HH_0(A)$ and the $k$-basis $e_1^\vee,\cdots,e_n^\vee$ of $HH_0(A^\op)$ is $C_A^T$.
Alternatively, by Proposition~\ref{Proposition-HH-GloDimFinAlg},
we know that $e_1,\cdots,e_n$ and $e_1^\vee,\cdots,e_n^\vee$ are $k$-bases of $HH_0(A)$ and $HH_0(A^\op)$ respectively.
Note that $\tr_{e_i A}(\id_{e_i A})=e_i$ and $\tr_{A e_j}(\id_{A e_j})=e_j^\vee$ for all $1\le i,j\le n$.
Taking $M, \phi, N$ and $\psi$ in Theorem~\ref{Theorem-Lefschetz-dga} to be $e_i A, \id_{e_i A}, Ae_j$ and $\id_{Ae_j}$ respectively, we obtain
$$\langle e_i, e_j^\vee\rangle
= \langle \tr_{e_iA}(\id_{e_iA}),\tr_{Ae_j}(\id_{Ae_j}) \rangle
= \tr_{e_iAe_j}(\id_{e_iAe_j})
= \dim e_iAe_j = (C_A)_{ji}$$
in $k$ for all $1\le i,j\le n$.
Thus the matrix of $\langle -,- \rangle$ under the bases is $C_A^T$.

Taking $N$ and $\psi$ in Theorem~\ref{Theorem-Lefschetz-dga} to be $Ae_i$ and $\id_{Ae_i}$ respectively,
we obtain $\tr_{Me_i}(\phi_i) = \tr_{M\ot^L_AAe_i}(\phi\ot^L_AAe_i) = \langle \tr_M(\phi),e_i^\vee \rangle = (X^T \cdot C_A^T)_i = (C_A\cdot X)_i$ for all $1\le i\le n$
where $X$ is the coordinate of $\tr_M(\phi)$ under the $k$-basis $e_1,\cdots,e_n$ of $HH_0(A)$.
So $\tv_M(\phi)=C_A\cdot X$, i.e., $X= C_A^{-1}\cdot \tv_M(\phi)$.
Thus $\tr_M(\phi) =(e_1,\cdots,e_n)\cdot C_A^{-1}\cdot \tv_M(\phi)$.
\end{proof}

\bigskip

\noindent{\bf Comparisons of the Lefschetz type formulas.}
Now we compare the Lefschetz type formulas in Theorem~\ref{Theorem-Lefschetz-Complex} (2) and (4)
with Petit's formulas in Theorem~\ref{Theorem-Lefschetz-dga} and Proposition~\ref{Proposition-Lefschetz-HochHom-dga}.

\medskip

{\it The homological Lefschetz type formula on complex level in Theorem~\ref{Theorem-Lefschetz-Complex} (2) is just
Petit's formula in Theorem~\ref{Theorem-Lefschetz-dga}
restricted to finite dimensional elementary algebras of finite global dimension:}
Let $A$ be a finite dimensional elementary algebra of finite global dimension,
and $\{e_1,\cdots,e_n\}$ a complete set of orthogonal primitive idempotents in $A$.
Without loss of generality (See Remark~\ref{Remark-Lefschetz-Complex-CohFin}),
let $\phi$ be a chain endomorphism of a bounded complex $M$ of finite dimensional projective right $A$-modules,
and $\psi$ a chain endomorphism of a bounded complex $N$ of finite dimensional projective left $A$-modules.
Petit's formula in Theorem~\ref{Theorem-Lefschetz-dga} is
$\tr_{M\ot^L_AN}(\phi\ot^L_A\psi) = \langle \tr_M(\phi),\tr_N(\psi) \rangle$.
By Proposition~\ref{Proposition-Pairing-CartanMatrix}, Proposition~\ref{Proposition-HochClass-TrCartan} and Lemma~\ref{Lemma-CartanMatrix-AlgebraTensorOp} (2),
its right hand side
$$\begin{array} {ll}
\langle \tr_M(\phi),\tr_N(\psi) \rangle
& \stackrel{\rm 2P}{=}  (C_A^{-1} \cdot \tv_M(\phi))^T \cdot C_A^T \cdot (C_{A^\op}^{-1} \cdot \tv_{N_{A^\op}}(\psi)) \\ [3mm]
& \stackrel{\rm L}{=}  (C_A^{-1} \cdot \tv_M(\phi))^T \cdot C_A^T \cdot (C_A^{-T} \cdot (\tv_N(\psi))^T) \\ [3mm]
& =  (\tv_M(\phi))^T \cdot C_A^{-T} \cdot (\tv_N(\psi))^T \\ [3mm]
& = \tv_N(\psi) \cdot C_A^{-1} \cdot \tv_M(\phi) \\ [3mm]
& \stackrel{\rm L}{=} \langle (\tv_N(\psi))^T, \tv_M(\phi)\rangle_{A^\op}.
\end{array}$$

\medskip

{\it The Hochschild homological Lefschetz type formula on complex level in Theorem~\ref{Theorem-Lefschetz-Complex} (4)
is just Petit's formula in Proposition~\ref{Proposition-Lefschetz-HochHom-dga}
restricted to finite dimensional elementary algebras of finite global dimension:}
Let $A$ be a finite dimensional elementary algebra of finite global dimension,
and $\{e_1,\cdots,e_n\}$ a complete set of orthogonal primitive idempotents in $A$.
Without loss of generality (See Remark~\ref{Remark-Lefschetz-Complex-CohFin}),
let $\phi$ be a chain endomorphism of a bounded complex $M$ of finite dimensional projective $A$-bimodules.
Petit's formula in Proposition~\ref{Proposition-Lefschetz-HochHom-dga} is
$\tr_{A\ot^L_{A^e}M}(A\ot^L_{A^e}\phi)=\langle \tr_{A_{A^e}}(\id_{A_{A^e}}),\tr_{_{A^e}M}(\phi) \rangle$.
Note that the canonical complete set of orthogonal primitive idempotents of $A^e$ is
$e_1\ot e_1,\cdots,e_1\ot e_n,\cdots,e_n\ot e_1,\cdots,e_n\ot e_n$.
So we have
$$\begin{array}{ll}
& \tv_{_{A^e}M}(\phi) \\[2mm]
= & (\tr_{(e_1\ot e_1)M}(\phi_{11}),\cdots,
\tr_{(e_1\ot e_n)M}(\phi_{1n}),\cdots,
\tr_{(e_n\ot e_1)M}(\phi_{n1}),\cdots,\tr_{(e_n\ot e_n)M}(\phi_{nn})) \\[2mm]
= & (\tr_{e_1M e_1}(\phi_{11}), \cdots,
\tr_{e_nM e_1}(\phi_{1n}),\cdots,
\tr_{e_1M e_n}(\phi_{n1}),\cdots,\tr_{e_nM e_n}(\phi_{nn})) \\ [2mm]
= & (\tv_{Me_1}(\phi_1),\cdots,\tv_{Me_n}(\phi_n)).
\end{array}$$
By Proposition~\ref{Proposition-Pairing-CartanMatrix}, Proposition~\ref{Proposition-HochClass-TrCartan}, Lemma~\ref{Lemma-CartanMatrix-AlgebraTensorOp} (3), Lemma~\ref{Lemma-Property-Trace} (3), and Theorem~\ref{Theorem-Lefschetz-dga}, the right hand side of Petit's formula in Proposition~\ref{Proposition-Lefschetz-HochHom-dga}
$$\begin{array}{ll}
  & \langle\tr_{A_{A^e}}(\id_{A_{A^e}}),\tr_{_{A^e}M}(\phi) \rangle \\ [4mm]
\stackrel{\rm 2P}{=} & (C_{A^e}^{-1} \cdot \tv_{A_{A^e}}(\id_{A_{A^e}}))^T \cdot C_{A^e}^T \cdot (C_{A^e}^{-T} \cdot (\tv_{_{A^e}M}(\phi))^T) \\ [4mm]
= & (\tv_{A_{A^e}}(\id_{A_{A^e}}))^T \cdot C_{A^e}^{-T} \cdot (\tv_{_{A^e}M}(\phi))^T \\ [4mm]
= & \tv_{_{A^e}M}(\phi) \cdot C_{A^\op\ot A}^{-1} \cdot \tv_{A_{A^e}}(\id_{A_{A^e}}) \\ [4mm]
\stackrel{\rm 2L}{=} & \tv_{_{A^e}M}(\phi) \cdot (C_A^{-T}\ot C_A^{-1}) \cdot \dv A_{A^e}\\ [4mm]
= & (\tv_{Me_1}(\phi_1),\cdots,\tv_{Me_n}(\phi_n)) \cdot
\begin{pmatrix} (C_A^{-T})_{11}\cdot C_A^{-1} & \cdots & (C_A^{-T})_{1n}\cdot C_A^{-1}\\
\vdots & \ddots & \vdots\\ (C_A^{-T})_{n1}\cdot C_A^{-1} & \cdots & (C_A^{-T})_{nn}\cdot C_A^{-1} \end{pmatrix} \cdot
\begin{pmatrix} \dv e_1 A\\ \vdots \\ \dv e_n A \end{pmatrix}\\ [10mm]
= &\sum\limits_{1\le i,j\le n} \tv_{M e_i}(\phi_i) \cdot (C_A^{-T})_{ij}\cdot C_A^{-1} \cdot \dv e_j A \\ [4mm]
\stackrel{\rm T}{=} & \sum\limits_{1\le i,j\le n} (C_A^{-T})_{ij} \cdot \tr_{e_jA \ot^L_A Me_i}(e_jA \ot^L_A\phi_i) \\ [5mm]
= & \sum\limits_{1\le i,j\le n} (C_A^{-T})_{ij} \cdot \tr_{e_jMe_i}(\phi_{ij}) \\ [5mm]
= & \tr(C_A^{-1}\cdot \tm\phi).
\end{array}$$

\medskip

Last, using the same strategy as above, it is not difficult to show that,
{\it when both formulas are restricted to three finite dimensional elementary algebras of finite global dimension,
the homological Lefschetz type formula on bimodule complex level in Theorem~\ref{Theorem-Lefschetz-BiModComplex} (2)
and Petit's formula in \cite[Theorem 5.8]{Pet13} coincide.}

\bigskip

\noindent {\footnotesize {\bf ACKNOWLEDGEMENT.} The authors are sponsored by Project 11971460 NSFC.}

\end{document}